\documentclass[11pt]{article}

\usepackage{amsmath}
\usepackage{bbm}
\usepackage{amssymb}
\usepackage{amsthm}
\usepackage{amscd}

\usepackage[english]{babel}
\usepackage[utf8]{inputenc} 
\usepackage[T1]{fontenc}

\usepackage{authblk}

\usepackage[T1]{fontenc}
\usepackage{lmodern}

\newcommand{\Addresses}{{
		\bigskip
		\footnotesize
		
	 \textsc{Département de mathématiques et applications, École Normale Supérieure, PSL Research University, 45 rue d'Ulm,
			Paris, France, 75005.}\par\nopagebreak
		\textit{E-mail address}, \texttt{tristan.ozuch-meersseman@ens.fr}
	}}

\newtheorem{thm}{Theorem}[section]

\newtheorem{lem}[thm]{Lemma}
\newtheorem{prop}[thm]{Proposition}
\newtheorem{cor}[thm]{Corollary}

\theoremstyle{definition}
\newtheorem{defn}{Definition}[section]

\newtheorem{rem}[thm]{Remark}
\newtheorem{note}[thm]{Note}

\DeclareMathOperator{\Rm}{\textup{Rm}}
\DeclareMathOperator{\R}{\textup{R}}

\DeclareMathOperator{\Ric}{\textup{Ric}}

\author{OZUCH Tristan}
\affil{École Normale Supérieure, PSL Research University}
\date{}
\title{Perelman's functionals on cones \\ Construction of type III Ricci flows coming out of cones}

\begin{document}
\maketitle 
\begin{abstract}
 In this paper, we are interested in conical structures of manifolds with respect to the Ricci flow and, in particular, we study them from the point of view of Perelman's functionals.
	
 In a first part, we study Perelman's $\lambda$ and $\nu$ functionals of cones and characterize their finiteness in terms of the $\lambda$-functional of the link. As an application, we characterize manifolds with conical singularities on which a $\lambda$-functional can be defined and get upper bounds on the $\nu$-functional of asymptotically conical manifolds.

 We then present an adaptation of the proof of Perelman's pseudolocality theorem and prove that cones over some perturbations of the unit sphere can be smoothed out by type III immortal solutions on the Ricci flow.
\end{abstract}

\tableofcontents

\section{Introduction}
 Just like in the study of a lot of geometric equations, a crucial aspect of the study of Ricci flows is the analysis of singularities, and in particular the possibility to continue the flow through some of them. One of the type of singularity people are interesting in at the moment is conical sigularities, the local geometry of some points of the manifolds is modeled on that of a Riemannian cone.

 There have been a lot of work on desingularizing such manifolds. This has been done thanks to expanding solitons coming out of cones, see \cite{gs}. A good picture of a Ricci flow continued through a conical singularity that was formed is given in \cite{FIK} where a gradient shrinking soliton shrinks into a cone (the singularity) that is then smoothed out by a gradient expanding soliton.

 As a consequence, the research of expanding solitons asymptotic to cones is currently a particularly topical subject (see \cite{SS,DMod,Dth}). Cones manifolds are therefore a subject of interest for the study of Ricci flows.
\\

 In this paper, we first study cone manifolds from the point of view of Perelman's $\lambda$ and $\nu$ functionals and characterize their finiteness and give several lower bounds for cones over some perturbations of the sphere or its quotients. We use this to characterize manifolds with conical singularities with finite $\lambda$-functional and give some upper bounds on the $\nu$-functional of asymptotically conical manifolds and manifolds with conical singularities.

 We then adapt the proof of Perelman's pseudolocality theorem to prove that cones over some perturbations of the unit sphere can be smoothed out by an immortal type III solution of the Ricci flow coming out of the cone that are asymptotic to the cone at all times. In some cases, these manifolds are gradient expanding solitons.
\subsection{Main definitions}
Let us start by giving some definitions of the mathematical objects we will be considering. The definitions will be minimal, for more developped explanations, see \cite{TheRF} for detailed notes on the Ricci flow or \cite{KleinLott} for notes focused on Perelman's papers.
\\

A \textit{Ricci flow} on a differential manifold $N$ on an interval $I$ is a family of Riemannian metrics $(g_t)_{t \in I}$ on $N$ ($t$ will be referred to as the \textit{time}) that are satisfying the following evolution equation :
\begin{equation}
\partial_tg_t = -2\Ric(g_t),
\end{equation}
where $\Ric(g_t)$ is the Ricci curvature associated to the metric $g_t$.
\\

 A \textit{Ricci soliton} is a particular case of a Ricci flow that is a fixed point of the Ricci flow, up to pull-back by diffeomorphisms and scaling of the metric (they are self similar solutions of the Ricci flow).

In other words, it means that if $(g_t)_t$ is a Ricci soliton, there exists a one-parameter family of diffeomorphisms $\zeta_t$ and a scaling factor $\gamma$ such that $g_t = (1+\gamma t)\zeta_t^*g_0$.
\begin{rem}
The factor is affine because the Ricci tensor is scale invariant while the metric is not.
\end{rem}

\begin{rem}
 Note that if $\zeta_t$ is generated by a vector field $-V$, being a Ricci soliton is equivalent to satisfying what is called the \textit{Ricci soliton equation} :
\begin{align}\label{ricci soliton equation}
\Ric + \mathcal{L}_V g_0 -\frac{\gamma}{2} g_0 = 0.
\end{align}
\end{rem}

 An \textit{expanding} soliton satisfies $\gamma>0$, a \textit{steady} soliton corresponds to $\gamma = 0$ and a \textit{shrinking} soliton corresponds to $\gamma<0$.
\\

 A \textit{cone}  over a \textit{link} $(N,g^N)$, noted \textit{C(N)} is defined as :
$$(C(N),g)=\left(\mathbb{R}^+\times N, dr^2+r^2g^N\right).$$
\begin{note}
Whenever we will be working on a cone, we will be noting $r$ the coordinate on the $\mathbb{R}^+$ factor and $dr^2$ the associated metric.

 Along this paper, we will consider $N$ of dimension \textbf{n}, that is, $C(N)$ of dimension \textbf{(n+1)} and $M$ a manifold of dimension \textbf{(n+1)}.
\end{note}

 We will call a Ricci flow \textit{nonsingular} if it is defined on an interval of the form $[t_0,+\infty)$ (for example, expanding and steady solitons are immortal).
\\

 We will call Ricci flow of \textit{Type III} on $M$ if it is defined for times $t\in[0,+\infty)$ and has controlled curvature in the following way :

 There exists $C>0$ such that for all $t$ :
 $$|\Rm|(.,t)\leqslant \frac{C}{t}.$$
 \begin{rem}
 It is the needed condition to take limits of blowdowns of Ricci flow (process described in the appendix), the usual process to construct expanding solitons.
 \end{rem}

We will say that a Ricci flow $(g_t)_{t\in (0,+\infty)}$ on $M = \mathbb{R}^+\times N$ is \textit{coming out of the cone $C(N)$} if we have the two following properties :

$\;\;$ 1)$\;\;$The metric space $(\mathbb{R}^+\times N\backslash (0,.), g_t)$ converges to the metric space $(C(N),d^{C(N)})$ in the Gromov-Hausdorff sense as $t\to 0$.

$\;\;$ 2)$\;\;$On $M\backslash\left\{(0,.)\right\}$, $g_t$ converges \textit{smoothly} to $dr^2+r^2g^N$ where $r(x) = d^{C(N)}(x,(0,.))$.

\subsection{Presentation of the main results}
\subsubsection{Characterization of cones with finite $\mu$ and $\nu$-functionals}
We define a notion of Perelman's $\mu$ and $\nu$-functionals on Riemannian cones and characterize their finiteness.

 In a lot of cases, the $\mu$-functional of a cone is equal to $-\infty$ and so is the $\lambda$-functional.

 A characterization of cones with infinite $\mu$ or $\lambda$-functional is given in terms of the $\lambda$-functional of its link, namely :
\begin{thm}
Given $N$ a compact $n$-dimensional ($n \geqslant 2$) Riemannian manifold, noting :
$$
\left\{
    \begin{array}{ll}
        \mu^{C(N)}(\tau):= \mu\left(\left(\mathbb{R}^+\times N,dr^2+r^2g^N\right),\tau\right),\\
        \nu^{C(N)}:= \nu\left(\mathbb{R}^+\times N,dr^2+r^2g^N\right),\\
        \lambda^{C(N)}:= \lambda\left(\mathbb{R}^+\times N,dr^2+r^2g^N\right),\\
        \lambda^{N}:= \lambda\left(N,g^N\right).
    \end{array}
\right.
$$

We have the following informations on Perelman's functionals on the cone depending on the link :
\begin{itemize}
\item For the $\mu$ and $\nu$-functionals, we have :
\\
For all $\tau>0$ :
\begin{align*}
&\mu^{C(N)}(\tau) = \nu^{C(N)} = -\infty,\;\text{  if and only if : }\;\lambda^N\leqslant (n-1).
\end{align*}

\item For the $\lambda$-functional, we have:

\[\lambda^{C(N)}= -\infty,\;\text{  if and only if : }\;\lambda^N< (n-1),\]
\indent and :
\[\lambda^{C(N)} = 0 ,\;\text{  if and only if : }\; \lambda^N\geqslant (n-1).\]
\end{itemize}
\end{thm}
\begin{rem}
This makes clear that it is possible to have $\lambda$ bounded while the minimum of the scalar curvature is arbitrarily negative (here we can have $\R_{min}=-\infty$ and $\lambda = 0$).
\end{rem}
\indent $\lambda^N>(n-1)\;$ is a quite strong condition on the positivity of the curvature that limits the possible topologies, and sizes of the link. (recall that : $\R_{min}\leqslant \lambda\leqslant \R_{av}$ on compact manifolds). 

The result is actually a log-Sobolev inequality on cones and the condition $\lambda^N>(n-1)$ can also be seen as a condition implying a dimension free log-Sobolev inequality on the link, since for example, $\Ric^N> \frac{n-1}{n}g^N$ implies $\lambda^N>(n-1)$ and $\Ric^N> \frac{n-1}{n}g^N$ implies a log sobolev inequality with constant : $1$ on N, see \cite{ca}.

\subsubsection{Perelman's functionals on manifolds with conical singularities and asymptotically conical manifolds}
As a consequence of the previous results on cone and as a justification of the usefulness of our definitions, we can characterize on which manifolds with conical singularities (see definition \ref{manifold with conical singularities}) it is possible to define a $\lambda$-functional on such manifolds.

\begin{cor}
	Let $(M^n,g)$ be a compact manifold with conical singularities (see definition \ref{manifold with conical singularities}) such that one singularity is modeled on a cone $C(N)$ on a section $N$ such that $\lambda^N < (n-2)$.
	
	Then $$\lambda^M = -\infty.$$
	
	Conversely, if each singularity at a $x_i$ is modeled on a cone $C(N_i)$ with a link $N_i$ such that $\lambda^N > (n-2)$, then, $$\lambda^M > -\infty.$$
\end{cor}

\begin{rem}
	It was proven by Wang and Dai, in Wang's PhD thesis, that the $\lambda$-functional was not infinite when $R^N>(n-2)$ for each conical singularity link. Since $\lambda^N\geqslant\min(R^N)$, we recover their result and have a precise threshold. 
	
	Note that with the definition (\ref{manifold with conical singularities}) of a manifold with conical singularities we chose, we cannot decide for the case $\lambda^N = (n-2)$. If there is a fast enough convergence to the conical model at each singularity, then $\lambda^N = (n-2)$ implies $\lambda^M > -\infty$.
\end{rem}

\begin{cor}
	For $(M^n,g)$, a manifold smoothly asymptotic to the cone $C(N)$ at infinity, if $\lambda^N\leqslant (n-2)$, then, $$\nu^M(g) = - \infty.$$
	
	If, $\lambda^N > (n-2)$, then, $$\nu^M \leqslant \nu^{C(N)}.$$
\end{cor}

\subsubsection{A global pseudolocality theorem}

\indent We give a simple condition implying that a manifold will generate a nonsingular type III Ricci flow by just checking its $\nu$-functional. The proof is an adaptation of the proof of Perelman's pseudolocality theorem :

\begin{prop}
For all $n\geqslant 3$, there exists $\eta_n>0$ (supposed optimal for the following property), such that, for any $n$-dimensional Riemannian manifold $(M,g)$ such that the Ricci flow starting at it exists for a short time :

If the manifold satisfies $$\nu(g)> -\eta_n,$$ then the Ricci flow starting at $M$ exists and is nonsingular for all time $t>0$ and we have the following estimate, there exists $\alpha(\nu(g))$ such that, for all $t>0$ :
\begin{align*}
|\Rm(g_t)|\leqslant \frac{\alpha}{t}.
\end{align*}
\indent In other words, the Ricci flow is nonsingular and of type III.
\begin{rem}
By the description of the high curvature regions in dimension 3 done by Perelman, we can get an explicit constant $\eta_3 = \inf_{\alpha>0}(\eta(\alpha,3)) = 1-\log 2$.
\end{rem}
\end{prop}

\subsubsection{Construction of type III solutions and expanding solitons coming out of some cones}
The global pseudolocality giving a way to ensure that a Ricci flow is of type III, we construct type III solutions of the Ricci flow asmptotic to some cones.

Realizing that some cones have a high $\nu$-functional, we prove that it is sometimes possible to smooth them out by manifolds (asymptotic to the initial cone) while keeping a large $\nu$-functional, so that we can apply the pseudolocality result of the previous part.
\\
\indent This smoothening process relies on the study of the renormalizations of the Ricci flow (applied to the link only) adapted to manifolds shrinking in finite time into a sphere. And the study of the $\mathcal{W}$-functional of perturbations of the unit sphere.
\\

 We get in particular the following result for some cones over perturbations of the unit sphere.
\begin{thm}
For all $n\geqslant 2$, there exists $\beta_1(n)<1<\beta_2(n)$ such that, for all metric $g^N$ on $\mathbb{S}^n$ satisfying the property noted $P(\beta_1,\beta_2)$ :
\begin{itemize}
\item \textsl{Positivity of the curvature} : The isotropic curvature when crossed with $\mathbb{R}^2$ is positive.
\item \textsl{$C^0$-closeness to the sphere} : 
$$\beta_1^2g^{\mathbb{S}^n}\leqslant g^N \leqslant \beta_2^2g^{\mathbb{S}^n}.$$
\item \textsl{Lower bound on the scalar curvature} :
$$\R^N\geqslant \frac{n(n-1)}{\beta_2^2}.$$
\end{itemize}

Then, there exists an immortal type III solution of the Ricci flow $(\mathbb{R}^{n+1},g_t)$ coming out of the cone 
$C(N) := (\mathbb{R}^+\times \mathbb{S}^n, dr^2 +r^2g^N)$ that stays asymptotic to it at all times.
\end{thm}

\begin{rem}
	Having a positive curvature operator actually implies the first condition.
\end{rem}

\subsection{Introduction to Perelman's functionals}
Let us define Perelman's $\mathcal{F}$, $\lambda$, $\mathcal{W}$, $\mu$ and $\nu$ functionals and their principal features that will be used throughout the paper. There is a lot of literature about these quantities (see in particular \cite{TheRF, KleinLott} and for the first introduction by Perelman \cite{Perentformula}).

\subsubsection{The $\mathcal{F}$ and $\lambda$ functionals}
Let us start by defining the $\mathcal{F}$-functional, which is the base functional to define the other ones.
\begin{defn}
We define the \textit{$\mathcal{F}$-functional} :
\begin{align*}
\mathcal{F}(\phi,g):=\int_N\left(|\nabla \phi|^2+\R \right)e^{-\phi}dv,
\end{align*}
where $\R$ is the \textit{scalar curvature} of $g$.
\\
\end{defn}
From this, we can define another important quantity :
\begin{defn}
We can define the \textit{$\lambda$-functional} as :
\begin{align*}
\lambda(g) := \inf_{\phi}\mathcal{F}(\phi,g),
\end{align*}
where the infimum is taken among smooth $\phi$ such that
$$\int_M e^{-\phi}dv = 1.$$
\begin{rem}
Note that it is also the first eigenvalue of the operator
$$-4\Delta +\R,$$
and is always between the least value of the scalar curvature $\R_{min}$ and its average $\R_{av}$.

\end{rem}

\end{defn}
The main feature of this functional is that the Ricci flow can be see as a gradient flow \emph{up to diffeomorphisms} for it :
\begin{thm}
If $g$ and $\phi$ are evolving according to :
\begin{equation}
	\left\{
	\begin{array}{ll}
	\partial_t g &= -2\Ric,\\
	\partial_t \phi &= -\Delta \phi+|\nabla \phi|^2-\R,
	\end{array}
	\right.\label{evolution equation F}
\end{equation}

then :
\begin{align*}
\partial_t \left(\mathcal{F}(g_t,\phi_t)\right) = 2\int_M|\Ric+Hess \phi|^2 e^{-\phi}dv\geqslant 0,
\end{align*}
and the monotonicity is strict unless $(g,\phi)$ is a \textit{gradient steady soliton}, which is a steady soliton for which the vector field $V$ is $\nabla \phi$.
\end{thm}
\begin{rem}
These functionals are invariant by diffeomorphism action.
\end{rem}
\subsubsection{The $\mathcal{W}$, $\mathcal{N}$, $\mu$ and $\nu$ functionals}
Let us now define the $\mathcal{W}$, $\mu$ and $\nu$ functionals that are the most useful functionals to study finite time singularities of the Ricci flow :
\begin{defn}
Let us give the formula right away :
\begin{itemize}
\item The \textit{$\mathcal{W}$-functional} is defined as :
\begin{align*}
\mathcal{W}(f,g,\tau) :&= \int_M\left[\tau\left(|\nabla f|^2+\R\right)+f-n\right]\left(\frac{e^{-f}}{(4\pi\tau)^{\frac{n}{2}}}dv\right) \\
&= \tau\mathcal{F}(\phi,g)+\mathcal{N}(\phi,g)-\frac{n}{2}\log(4\pi\tau)-n,
\end{align*}
where $\mathcal{N}(\phi,g) := \int_M \phi e^{-\phi}dv$ is \textit{Nash entropy} functional,
\\
wnd where $\phi := f+\frac{n}{2}\log(4\pi\tau)$ satisfies $\int e^{-\phi} = 1$.

\item The \textit{$\mu$-functional} (often referred to as \textit{entropy}) is defined by :
\begin{align*}
\mu(g,\tau) := \inf_f \mathcal{W}(f,g,\tau),
\end{align*}
where the infimum is taken among the smooth $f$ such that 
$$\int\frac{e^{-f}}{(4\pi\tau)^{\frac{n}{2}}}dv =1.$$
\item The \textit{$\nu$-functional} is defined by :
\begin{align*}
\nu(g) := \inf_{\tau>0} \mu(g,\tau).
\end{align*}
\end{itemize}
\end{defn}
These quantities can again be used to see the Ricci flow as a gradient flow \textit{up to diffeomorphisms and change of scale} :
\begin{thm}
If $g$ and $\phi$ are evolving according to :
\begin{equation}
	\left\{
	\begin{array}{ll}
	\partial_t g &= -2\Ric,\\
	\partial_t f &= -\Delta f+|\nabla f|^2-\R+\frac{n}{2\tau},\\
	\partial_t \tau &= -1,
	\end{array}
	\right.\label{evolution equation W}
\end{equation}

then :
\begin{align*}
\partial_t \left(\mathcal{W}(f_t,g_t,\tau(t))\right) = 2\tau\int_M\left|\Ric+Hess f-\frac{1}{2\tau}g\right|^2 \frac{e^{-f}}{(4\pi\tau)^{}\frac{n}{2}}dv\geqslant 0,
\end{align*}
and the monotonicity is strict unless $(g,\phi,\tau)$ is a \textit{gradient shrinking soliton} which is a shrinking soliton for which the vector field $V$ is $\nabla f$ and $\gamma = -\frac{1}{2\tau}.$
\end{thm}
\begin{rem}
These functionals are invariant by diffeomorphism action and also by \textit{parabolic scaling} : $(g,\tau)\mapsto(\alpha g, \alpha \tau)$ for $\alpha>0$.
\end{rem}
\begin{note}
We will sometimes use the following change of functional :
$$u^2 = \frac{e^{-f}}{(4\pi\tau)^{\frac{n}{2}}}.$$
For purpose of notation, we will also note $\mathcal{W}$ functional applied to the change of function $u$, for a $n$-dimensional manifold $M$, the expression is the following :
\begin{align}
\mathcal{W}^M(f,g,\tau) =& \int_M [\tau(|\nabla f|^2 + \R)+f-n]\left(\frac{e^{-f}}{(4\pi\tau)^{\frac{n}{2}}}dv\right) \nonumber\\
=&\int_M [\tau(4|\nabla u|^2+\R u^2)-u^2\log(u^2)]dv-\frac{n}{2}\log(4\pi\tau)-n.\label{corresp f u}
\end{align}
\end{note}

\subsection{$\mathcal{W}$ and $\mu$-functionals on cones}
We will define the functionals as the usual expression on the smooth part of the cone, that is the cone without its tip. This definition is justified by the fact that given a manifold with conical singularities or asymptotic to some cone, the values of the functionals on the cone give informations on their value on the manifold.

 All along the paper, we will be interested in cones from the point of view of Ricci flows, and we will in particular look at them through Perelman's functionals which take a particular form.
\begin{note}
A few basic facts about to keep in mind about these functionals :
\begin{itemize}
\item For a cone, for all $\tau>0$, $\mu(g,\tau) = \nu(g)$ because the cone is scale invariant while the $\mu$-functional is invariant by parabolic scaling.
\item The $\mu$-functional of the Euclidean cone : $\mathbb{R}^{n+1} = C(\mathbb{S}^n)$ is vanishing, and this is the only cone with this property (in other cases, $\nu^{C(N)}<0$).
\\
The minimizers at $\tau$ are Gaussians corresponding to the potential : $\frac{|.-x_0|^2}{4\tau}$ for any $x_0$ on the manifold.
\end{itemize}
\end{note}
\begin{rem}
Even if it would be possible to make every computation at $\tau = 1$ thanks to the first point, we will make every computation at a given $\tau > 0$ to emphasize its independance and because they will also be used for warped product later in this paper.
\end{rem}
\subsubsection{The basic formula}
On a cone, the formula for the entropy has a particular expression which is given by next lemma. We will refer to it as "basic".
\\

The definition we will take is the usual formula for the cone with the tip excluded which is an incomplete manifold. This definition is motivated by the applications to manifolds with conical singularities and manifolds asymptotically conical we present.
\begin{lem}
\begin{align}
\mathcal{W}^{C(N)}(f,g^{C(N)},\tau) = \int_0^\infty\int_N&\left[\tau\left((\partial_r f)^2+\frac{|\nabla^N f|^2+(\R^N-n(n-1))}{r^2}\right)\right.\nonumber\\ &\left. \Huge+f-(n+1)\right]\left(\frac{e^{-f}}{(4\pi\tau)^{\frac{n+1}{2}}}r^ndv dr \right).
\label{basic}
\end{align}
\begin{proof}
On a cone over a manifold $N$ : $(C(N),g^{C(N)}) = (\mathbb{R}^+\times N,dr^2+r^2g^N)$, with the coordinates $(r,x)\in \mathbb{R}^+\times N$ such that $d^{C(N)}\left((0,x),(r,x)\right)=r$.

 We have the following formulas for any functional $f$ on the cone :
\begin{itemize}\label{quant cone}
\item $\Ric =
\left[ \begin{array}{cc}
0 & 0  \\
0 & \Ric^N-(n-1)g^N  \end{array} \right]$ (it vanishes except in the direction of the link),
\item $\R = \frac{\R^N-(n(n-1))}{r^2}$,
\item $|\nabla f|^2 = (\partial_r f)^2+\frac{|\nabla^N f|^2}{r^2}$,
\item $dv^{C(N)} = r^ndv^Ndr$ (we will note $dv$ the volume form on $N$ in the following).
\end{itemize}
Now, if we plug this in the expression of the entropy, we get :
\begin{align*}
\mathcal{W}^{C(N)}&(f,g,\tau) := \int_{C(N)} \left[\tau(|\nabla^N f|^2+\R)+f-(n+1)\right](4\pi\tau)^{-\frac{n+1}{2}}e^{-f}dv\\
&= \int_0^\infty\int_N \left[\tau\left((\partial_r f)^2+\frac{|\nabla^N f|^2+(\R^N-n(n-1))}{r^2}\right)+f-(n+1)\right](4\pi\tau)^{-\frac{n+1}{2}}e^{-f}r^ndv dr.
\end{align*}
\end{proof}
\end{lem}

\section{Lower bounds on the $\lambda$ and $\nu$-functionals of cones \\ Characterization of $N$ such that $\nu^{C(N)} > -\infty$ and $\lambda^{C(N)} > -\infty$}
Let us look more carefully at the $\mu$-functional of cones through a separation of variables formula and characterize, thanks to it, closed manifolds $N$ for which $\nu^{C(N)} = -\infty$.
\subsection{A separation of variables formula}
Let us introduce a separation of variables formula that is natural if one wants to emphasize the fact that the $\mathcal{W}$ functional of a cone is closely related to the behavior of the $\mathcal{W}$-functional on its link. 
\\
\indent This separation of variables will be particularly interesting to get lower bound on the cone $\mu$-functional.
\begin{lem}
On a cone $(C(N),g) = (\mathbb{R}^+\times N, dr^2+r^2g^N)$ with the coordinates $(r,x)\in \mathbb{R}^+\times N$ such that for all $x\in N$, $d^{C(N)}\left((0,x),(r,x)\right)=r$.
\\
We can define the following separation of variables :
\\
\\
$\forall f : \mathbb{R}\times N \to \mathbb{R}$,

$\exists !$ $(\tilde{f}, a)$, $\tilde{f}:C(N)\to \mathbb{R}$ and $a:\mathbb{R}^+\to \mathbb{R}$, such that for all $r>0$,

\begin{itemize}
\item $f(r,.) =  \tilde{f}(r,.) + a_r$ which gives $e^{-f} = e^{-a_r}e^{-\tilde{f}}$,
\item $\int_N(4\pi\tau \textbf{r}^{-2})^{-\frac{n}{2}} e^{-\tilde{f}}dv = 1$ (\textbf{notice the $r^{-2}$}),
\item $\int_0^\infty (4\pi\tau)^{-\frac{1}{2}}e^{-a_r}dr = 1$.
\end{itemize} 
Thanks to these, we can rewrite the expression as :
\begin{align}
\mathcal{W}^{C(N)}(f,g,\tau) &= \int_0^\infty \left[\mathcal{W}^{N}\left(\tilde{f},g^N,\frac{\tau}{r^2}\right)-n(n-1)\frac{\tau}{r^2}\right]\left(\frac{e^{-a_r}}{(4\pi\tau)^{\frac{1}{2}}} dr\right) \nonumber\\
&+\int_N \left(\int_0^\infty \left[\tau(\partial_r (\tilde{f}(r,.) + a_r))^2+a_r-1\right]\left(\frac{e^{-a_r}}{(4\pi\tau)^{\frac{1}{2}}} dr\right)\right)\left(\frac{e^{-\tilde{f}}}{\left(4\pi\tau r^{-2}\right)^{\frac{n}{2}}} dv\right).
\label{separation}
\end{align}
\begin{rem}
In the Euclidean cone, for a Gaussian centered at the origin (which is a minimizer of $\mathcal{W}^{\mathbb{R}^{n+1}}$), it gives :
$$\tilde{f} = \log\left(\frac{r^nvol(\mathbb{S}^n)}{(4\pi\tau)^{\frac{n}{2}}}\right),$$
and
$$a_r = \frac{r^2}{4\tau} - \log\left(\frac{r^nvol(\mathbb{S}^n)}{(4\pi\tau)^{\frac{n}{2}}}\right).$$
\end{rem}

\begin{proof}
By \eqref{basic}, we have :
\begin{align*}
\mathcal{W}^{C(N)}(f,g,\tau) = \int_0^\infty\int_N &\left[\tau\left((\partial_r f)^2+\frac{|\nabla^N f|^2+(\R^N-n(n-1))}{r^2}\right)\right.\\
&\left.+f-(n+1)\right]\frac{e^{-f}}{(4\pi\tau)^{\frac{n+1}{2}}}r^ndv dr .
\end{align*}
Now, if we define $a_r$ such that $\frac{e^{-a_r}}{4\pi\tau} = \int_N \frac{r^ne^{-f}}{(4\pi\tau)^{\frac{n}{2}}}dv$, and then define $\tilde{f} = f-a_r$. We get the existence of such a separation of variables, the unicity can also be checked.
\\
\indent Noting that : 
\\
$$r^n(4\pi\tau)^{-\frac{n+1}{2}} = (4\pi\tau r^{-2})^{\frac{n}{2}}(4\pi\tau)^{-\frac{1}{2}},$$
and separating the $\tilde{f}$ terms and the $a$ terms, we get :
\begin{align*}
\mathcal{W}^{C(N)}(f,g,\tau) =& \int_0^\infty \left(\int_N\left[\frac{\tau}{r^2}\left(|\nabla^N \tilde{f}|^2+(\R^N-n(n-1))\right)+\tilde{f}-n\right]\frac{e^{-\tilde{f}}}{\left(4\pi\tau r^{-2}\right)^{\frac{n}{2}}}dv\right)\left(e^{-a_r}(4\pi\tau)^{-\frac{1}{2}} dr\right) \\
&+\int_N \int_0^\infty [\tau(\partial_r f)^2+a_r-1]\left(\frac{e^{-a_r}}{(4\pi\tau)^{\frac{1}{2}}} dr\right) \left(\frac{e^{-\tilde{f}}}{\left(4\pi\tau r^{-2}\right)^{\frac{n}{2}}} dv\right),
\end{align*}
which is the wanted result.
\end{proof}
\end{lem}
\subsection{Lower bound on the cone $\mu$-functional}
We are interested in getting a lower bound for a cones $\mu$-functional. It becomes obvious from the separation of variables formulas that the $\mu$-functional is deeply linked to the behavior of the $\mathcal{W}$-functional on the link, and this \textit{at all scales} (because of the $\frac{\tau}{r^2}$ taking every single value in $(0,+\infty]$).

 Looking at this quantity at all scales is uncommon, as it is usually considered at very particular scales such as the time remaining to a possible singularity time.
\\

 It is known that this functional is vanishing as $\tau \to 0$ for smooth manifolds. And as $\mathcal{W}(..\tau) = \tau \mathcal{F}+\mathcal{N}-\frac{n}{2}\log \tau +C(n)$, it is natural to think that for large $\tau$, a minimizer will get closer and closer to a minimizer of $\mathcal{F}$. We will give a description of the behavior of this functional and its minimizers in the next parts.

\subsubsection{Study of $\tau\mapsto \mu(g^N,\tau)$ - Lower bounds at large $\tau$}
Here we are going to get sharp lower bounds on the left derivative of $\tau\mapsto\mu^N(g_0,\tau)$ that are only attained in the precise case when a minimizing function of $\mathcal{W}^N(.,g_0,\tau)$ is also a minimizer of $\mathcal{F}^N$.
\begin{lem}
$\tau\mapsto\mu(g_0,\tau)$ when $\lambda^N>-\infty$ is upper semicontinuous,

 and :
\begin{align}
\lim_{\tau_2\to\tau_1}\left(\frac{\mu(g_0,\tau_1)-\mu(g_0,\tau_2)}{\tau_1-\tau_2}\right)\geqslant \lambda^N-\frac{n}{2\tau_1}.
\end{align}
\indent This is sharp as the last inequality gets arbitrarily glose to being an equality when $\tau_1\to\infty$ since we have :

 For any compact manifold $N$, and $f_k$ minimizing of $\mathcal{W}^N(.,g_0,\tau_k)$ tends to a minimizer of $\mathcal{F}^N(.,g_0)$ in $H^1(N,g_0)$ as $\tau_k\to +\infty$.
\begin{proof}
See Appendix A.2. 
\end{proof}
\end{lem}
\begin{rem}
These estimates are sharp and actually attained for all $\tau$ larger than the extinction time for positively curved Einstein manifolds (see appendix B.1) :

\end{rem}

As a direct corollary, we get a control of the asymptotic behavior of the entropy as $\tau$ tends to infinity. 

\begin{cor}\label{borne mu}
For all $N$ closed Riemannian manifold of dimension $n$, there exists $A\in \mathbb{R}$ such that :
\begin{align}
\mu^N(\tau,g^N) \geqslant \tau\lambda^N-A-\frac{n}{2}\log_+(\tau)\label{largetau}
\end{align}
where $\log_+$ is the positive part of the $\log$.

Moreover, if there is a $T_N<+\infty$ (in particular if $\lambda^N > 0$) such that $\mu^N(T_N)=\nu^N$, the we have the sharper control :

$\forall \tau\geqslant T_N$,
\begin{align*}
\mu^N(\tau,g^N) \geqslant \mu^N(T_N)+(\tau-T_N)\lambda^N-\frac{n}{2}\log\left(\frac{\tau}{T_N}\right).
\end{align*}
\begin{rem}
Here, there is equality if and only if $N$ is a positively curved Einstein manifold.
\end{rem}
\end{cor}

\subsection{$\mu^{C(N)} = -\infty$ if and only if $\lambda^N \leqslant (n-1)$}
In this section, still interested in the behavior of the $\mu$-functional of a cone, we realize that in some cases the entropy of a cone is $-\infty$. We are going to give a characterization of compact $n$-manifolds that generate cones of infinite $\mu$-functional.
\\

The $\mu$-functional being very negative is usually associated to the \textit{collapsedness} of some region of the manifold (because we usually assume lower bounds on the scalar curvature). The best example probably being the proof of the Perelman's local noncollapsedness theorem along the Ricci flow (see \cite{KleinLott}, section 13). 
\\

 But here, if the cone is not flat, the curvature blows up close to the tip, and the condition implying a $\mu$-functional being infinite is actually linked to a negative enough curvature on the link and not to the collapsnedness of some region.

 This is characterized by the value of Perelman's $\lambda$ on the link. Namely, we are going to prove :
\begin{thm}
We have :
\begin{itemize}
\item For the $\mu$ and $\nu$ functionals :
$$\forall \tau>0 \text{,  }\;\mu^{C(N)}(\tau)\; = \;\nu^{C(N)}\; = \;-\infty,$$
if and only if :
$$\lambda^N \leqslant (n-1).$$
\item For the $\lambda$-functional :

$$\lambda^{C(N)} = -\infty,$$
if and only if
$$\lambda^N<(n-1).$$
Moreover, if $\lambda^N \geqslant (n-1) $, then $\;\lambda^{C(N)} = 0$.
\begin{rem}
Note that the fact that the two possible values for $\lambda^N$ are $0$ and $-\infty$ is not surprising since the cone is scaling invariant while the $\mathcal{F}$-functional is not.
\\

(Here, when $(n-1)\leqslant \lambda^N < n(n-1)$, we have $\R_{min} = -\infty\;$ while $\lambda^N = 0$).
\end{rem}
\end{itemize}

\end{thm}

\indent This means that the links that have cones of finite $\mu$ correspond to quite positively curved manifolds, as $\lambda^N > (n-1)$ implies in particular that the average scalar curvature is stricly larger than $(n-1)$. 
\\
\indent This is a strong condition that limits the possible topologies of the manifold. in particular, in dimension $2$, it must be diffeomorphic to a sphere, in dimension $3$, it must be diffeomorphic to $\mathbb{S}^3$ or $\mathbb{S}^2\times \mathbb{S}^1$.

\begin{rem}
\indent At the same time, it is not a very strong condition compared to a lower bound on the scalar curvature. As the theorem in particular implies that there are cones with $\R_{min}= -\infty$ and $\lambda = 0$.
\\
\\
\indent This condition is implied by $\Ric^N\geqslant \frac{n-1}{n}g^N$, which is the kind of condition that implies dimension free log-Sobolev inequalities on $n$-dimensional manifolds (if $\Ric\geqslant K>0$, then there is a log-Sobolev inequality with a constant $\left[K\frac{n}{n-1}\right]$).
\\
\indent It is not surprising that the condition depends on the $\lambda$-functional of the link. 
\\
\indent Morally, the geometry of the cone far away from the tip gets mild (locally nearly flat), and this controls the functionals. So the degenerate behavior should come from the region around the tip of the cone.
\\
\indent The regions close to the tip (for $r$ small) correspond to large values of $\frac{\tau}{r^2}$ in the $\mathcal{W}^N$ term of the separation of variables expression, which means that we are looking at the link for large $\tau$, and we have just seen that the behavior of the $\mu$-functional at large $\tau$ is ruled by the $\lambda$-functional value.
\\
\end{rem}

\textbf{For spheres}, the condition is :
$$\forall \tau,\mu^{C(\beta\mathbb{S}^n)}(\tau) = \nu^{C(\beta\mathbb{S}^n)} = -\infty,$$
if and only if :
$$\beta \geqslant \sqrt{n},$$
and we see that these cones are actually the least collapsed among the cones over spheres, and we can see that the quantity that makes the entropy infinite is mostly the low scalar curvature.

\subsubsection{If $\lambda^N\leqslant (n-1)$, then $\nu^{C(N)} = -\infty$}
Let us start the proof by the implication involving the least intermediate results.
\begin{prop}
We have the first implications :
\begin{itemize}
\item If $\lambda^N \leqslant (n-1)$, then $\nu^{C(N)} = -\infty$.
\item If $\lambda^N < (n-1)$, then $\lambda^{C(N)}=-\infty$.
\end{itemize}
\end{prop}
\begin{proof}
Let $N$ be a $n$-dimensional closed manifold and $\tau>0$.
\\
\indent Let us define :
$$K :=\lambda^N - n(n-1) =\inf_{\int_N u^2 = 1}\int_N \left[(\R^N-n(n-1))u^2+4|\nabla u|^2\right]dv.$$
\indent We want to prove that if $K$ is lower than $-(n-1)^2$ (equivalent to $\lambda^N\leqslant (n-1)$), then there is a function $v^2 = \frac{e^{-f}}{(4\pi\tau)^{\frac{n+1}{2}}}$ such that $\mathcal{W}(f,g,\tau)$ is $-\infty$.
\\

 It is enough to find a function in $H^1$ for which it is infinite.

 Let us construct such a function.
\subparagraph{1) Let us consider $v = \left(\frac{b}{r^a}\tilde{u}\right) \chi_{[0,2r_0]}$,} where $\tilde{u}$ is a minimizer of Perelman's $\mathcal{F}$ functional, that is :
$$\int_N \left[\left(\R^N-n(n-1)\right)\tilde{u}^2+4|\nabla \tilde{u}|^2\right]dv = K,$$
and $$\int_N \tilde{u}^2 dv =1,$$

where $\chi_{[0,r_0]}$ is a cutoff function equal to $1$ in $[0,r_0]$ and with support in $[0, 2r_0]$ and a first derivative of order $\frac{1}{r_0}$.

and for some real numbers $a$, $b>0$ and $r_0>0$ such that $\int_{C(N)}v^2 = 1$. 

 The goal is to prove that we can choose these numbers to get $\mathcal{W}(f,g^{C(N)},\tau) = -\infty$.
\\

Let us first note that the part of the function that corresponds to the values in $[r_0,2r_0]$ is of finite influence, so it will be enough to look at the functional on $[0,r_0]$.

On $[r_0,2r_0]$, the derivative is bounded by $$\left|\partial_r \left(\frac{b}{r^a} \chi_{[0,2r_0]}\right)\right|\leqslant\frac{ab}{r_0^{a+1}} + \frac{b}{r_0^{a+1}}\mathcal{O}(1).$$

The Nash entropy term is also bounded on this domain because the function is bounded. We can therefore restrict our attention to the interval $[0,r_0]$.
\\

 By plugging this expression in the $\mathcal{W}$-functional, we get :
$$\mathcal{W}^{C(N)}(f,g,\tau) = \int_0^{r_0} \left[\tau \left(b^2r^{-2a-2+n}(4a^2+K)\right)+\left(2b^2r^{-2a+n}a\log r\right)\right]dr -\frac{(n+1)}{2}\log(4\pi\tau)+(n+1).$$
\paragraph{a. Making the first term negative}
:
\\
For the first term $\tau \left(b^2r^{-2a-2+n}(4a^2+K)\right)$ (that we want to make much bigger than the others around $0$) to be negative, we want $4a^2+K<0$ that is :
\begin{align}
a<\frac{\sqrt{-K}}{2}.\label{cond1}
\end{align}
\paragraph{b. Making the remaining terms finite while the first one is infinite}
:
\\
Now, to get an arbitrarily negative entropy, we would like to have $r^{n-2a}\log r$ integrable while $r^{n-2a-2}$ is not.
\begin{rem}
Note that this will also imply that our function is indeed in $L^2$.
\end{rem}
We want $n-2a>1$ while $n-2a-2\leqslant 1$.
\\

 This way, the first term produces an infinite negative quantity while the second is finite.
\\
We want $a$ to satisfy :
\begin{align}
\frac{n-1}{2}\leqslant a < \frac{n+1}{2}.\label{cond2}
\end{align}
\paragraph{c. Values for which \eqref{cond1} and \eqref{cond2} are consistent}
:
\\
These two conditions are consistent for some $a$ ($<\frac{n+1}{2}$) if and only if :
$$\frac{n-1}{2}<\frac{\sqrt{-K}}{2},$$ that is :
$$K< -(n-1)^2.$$

In conclusion, if 
$$\lambda^N< n(n-1)-(n-1)^2 =(n-1),$$ 
then $$\mu^{C(N)}=-\infty.$$
\begin{rem}
Here we have actually made $\mathcal{F}^{C(N)}$ arbitrarily negative (while keeping $\mathcal{N}^{C(N)}$ bounded). Hence the second part of the proposition.
\end{rem}
\paragraph{2) Remains the equality case $\lambda^N = (n-1)$, that is $K = -(n-1)^2$}
:
\\
\indent In this case, let us choose $a = \frac{\sqrt{-K}}{2} = \frac{n-1}{2}$. The first term vanishes and we are left with :
$$\int_0^{r_0} 2b^2r^{-2a+n}a\log r dr - C(n,\tau)= (n-1)b^2\int_0^{r_0} r^{-1}\log r dr - C(n,\tau) = -\infty,$$ because $\frac{\log r}{r}$ is not integrable and negative around zero.
\\

If $\lambda^N\leqslant (n-1)$, then $\mu^{C(N)} = -\infty$.
\begin{rem}
Here we have actually made $\mathcal{N}^{C(N)}$ arbitrarily negative while the first term vanishes.
\end{rem}
\end{proof}

\subsubsection{If $\lambda^N>(n-1)$, then $\mu^{C(N)}>-\infty$}
Let us now present the more challenging other implication. 
\begin{prop}
We have the following other implication :

For the $\nu$-functional :
\begin{itemize}
\item If $\lambda^N>(n-1)$, then, $\nu^{C(N)}>-\infty$.
\end{itemize}

For the $\lambda$-functional :
\begin{itemize}
\item If $\lambda^N\geqslant (n-1)$, then, $\lambda^{C(N)} = 0$.
\end{itemize}
\end{prop}
\begin{proof}
We will focus on the more challenging first statement and point out from which part of the proof it is possible to deduce the second part.

Let $N$ be a $n$-dimensional manifold such that $\lambda^N>(n-1)$, and consider $\tau>0$ and $f:C(N)\to \mathbb{R}$ a smooth function such that :
$$\int_{C(N)}\frac{e^{-f}}{(4\pi\tau)^{\frac{n+1}{2}}}dv^{C(N)} = 1.$$

We want to bound the $\mathcal{W}^{C(N)}(f,g,\tau)$ from below.
\\

Let us start by explaining why it is enough to bound the part of the functional corresponding to the ball of radius $C\sqrt{\tau}$ around the tip of the cone. The intuition behind this is that at larger distance, the geometry becomes milder and the $\mathcal{W}$-functional cannot take too negative values.
\begin{lem}\label{control a grande distance}
There exists $C >0$, depending on the geometry of $N$, such that :
\begin{align*}
\int_{C\sqrt{\tau}}^{+\infty}\int_N\left[\tau(|\nabla f|^2+\R)+f-(n+1)\right]\left(\frac{e^{-f}}{(4\pi\tau)^{\frac{n+1}{2}}}r^ndvdr\right) &\geqslant -1\\
&>-\infty
\end{align*}
\begin{proof}
	For any smooth manifold $(N,g)$, $ \lim_{\tau\to 0} \mu^N(\tau,g) = 0 $.
	
	In particular, there exists $\frac{1}{2n(n-1)}>\tau_0>0$, such that for all $0 \leqslant \tau\leqslant\tau_0$, 
	\\
	we have $\mu^N(\tau,g)\geqslant -\frac{1}{2}$.
	
	Now, we can use the separation of variables formula for the $\mathcal{W}$-functional on the cone without a ball of radius $\sqrt{\frac{\tau}{\tau_0}}$ (chosen to have $0<\frac{\tau}{r^2}\leqslant \tau_0$) :
	
	\begin{align*}
	\mathcal{W}^{C(N)}(f) =& \int_{\sqrt{\frac{\tau}{\tau_0}}}^\infty \left[\mathcal{W}^{N}\left(\tilde{f},g,\frac{\tau}{r^2}\right)-n(n-1)\frac{\tau}{r^2}\right]\left(\frac{e^{-a_r}}{(4\pi\tau)^{\frac{1}{2}}} dr\right) \\
	&+\int_N \int_{\sqrt{\frac{\tau}{\tau_0}}}^\infty [\tau(\partial_r (a_r+\tilde{f}))^2+a_r-1]\left(\frac{e^{-a_r}}{(4\pi\tau)^{\frac{1}{2}}} dr\right)\left(\frac{e^{-\tilde{f}}}{\left(4\pi\tau r^{-2}\right)^{\frac{n}{2}}} dv\right)\\
	\geqslant &  \int_{\sqrt{\frac{\tau}{\tau_0}}}^\infty \left[-\frac{1}{2}-n(n-1)\tau_0\right]\left(\frac{e^{-a_r}}{(4\pi\tau)^{\frac{1}{2}}} dr\right) \\
	&+\int_N \int_{\sqrt{\frac{\tau}{\tau_0}}}^\infty [\tau(\partial_r (a_r+\tilde{f}))^2+a_r-1]\left(\frac{e^{-a_r}}{(4\pi\tau)^{\frac{1}{2}}} dr\right)\left(\frac{e^{-\tilde{f}}}{\left(4\pi\tau r^{-2}\right)^{\frac{n}{2}}} dv\right)\\
	\geqslant & -1\\ 
	&+ \int_N \int_{\sqrt{\frac{\tau}{\tau_0}}}^\infty [\tau(\partial_r (a_r+\tilde{f}))^2+a_r-1]\left(\frac{e^{-a_r}}{(4\pi\tau)^{\frac{1}{2}}} dr\right)\left(\frac{e^{-\tilde{f}}}{\left(4\pi\tau r^{-2}\right)^{\frac{n}{2}}} dv\right).
	\end{align*}

The last term is non negative, as we shall see in more details in the next part of the proof, thanks to the lemma \ref{lower bound radial term} to get rid of the $\tilde{f}$ term followed by the lemma \ref{log sobolev radial} to control the resulting integrand. 

\begin{rem}	
	The lemma \ref{log sobolev radial} can only be directly used with functions such that $$\int\frac{e^{-f}}{(4\pi\tau)^{\frac{n+1}{2}}} = 1,$$ but since, for all $c\in \mathbb{R}$,
	$$\mathcal{W}(f+c,g,\tau) = e^{-c}\mathcal{W}(f,g,\tau) + ce^{-c},$$
	we can directly relate our result to other values of the integral. 
	
	In particular, if $\int\frac{e^{-f}}{(4\pi\tau)^{\frac{n+1}{2}}} \leqslant 1$ and $\mathcal{W}(f+c,g,\tau) \geqslant 0$ for a constant $c\leqslant 0$ such that $\int\frac{e^{-(f+c)}}{(4\pi\tau)^{\frac{n+1}{2}}} = 1$, then we also have $\mathcal{W}(f,g,\tau) \geqslant 0$.
	
\end{rem}
\end{proof}
\end{lem}

\indent As a consequence, it is then enough to find a lower bound for the part of the $\mathcal{W}$-functional at scale $\tau$ corresponding to the truncated cone $\left([0,C\sqrt{\tau}], dr^2+r^2g^N\right)$.
\begin{rem}
In this part of the manifold, the important quantity is the $\mathcal{W}$-functional on the link at large scales. That is the reason why the most influent quantity is $\lambda^N$. 
\end{rem}
\paragraph{0) We have the following expression of the entropy}
:
\\
By the separation of variables \eqref{separation} :
\begin{align*}
\mathcal{W}^{C(N)}(f) =& \int_0^\infty \left[\mathcal{W}^{N}\left(\tilde{f},g,\frac{\tau}{r^2}\right)-n(n-1)\frac{\tau}{r^2}\right]\left(\frac{e^{-a_r}}{(4\pi\tau)^{\frac{1}{2}}} dr\right) \\
&+\int_N \int_0^\infty [\tau(\partial_r (a_r+\tilde{f}))^2+a_r-1]\left(\frac{e^{-a_r}}{(4\pi\tau)^{\frac{1}{2}}} dr\right)\left(\frac{e^{-\tilde{f}}}{\left(4\pi\tau r^{-2}\right)^{\frac{n}{2}}} dv\right).
\end{align*}
\paragraph{1) Let us first bound the second term}
$$\int_N \int_0^\infty [\tau(\partial_r( a_r+\tilde{f}))^2+a_r-1]\left(\frac{e^{-a_r}}{(4\pi\tau)^{\frac{1}{2}}} dr\right)\left(\frac{e^{-\tilde{f}}}{\left(4\pi\tau r^{-2}\right)^{\frac{n}{2}}} dv\right)$$ by a nicer term without $\tilde{f}$ in the $\partial_r$ term :
\begin{lem}\label{lower bound radial term}
We have the following lower bound :
\begin{align*}
\int_N \int_0^\infty [\tau(\partial_r( a_r+\tilde{f}))^2+a_r-1]&\left(\frac{e^{-a_r}}{(4\pi\tau)^{\frac{1}{2}}} dr\right)\left(\frac{e^{-\tilde{f}}}{\left(4\pi\tau r^{-2}\right)^{\frac{n}{2}}} dv\right)\\
&\geqslant \mathcal{W}^{\mathbb{R}^+}\left(\left(a_r-\frac{n}{2}\log\frac{\tau}{r^2}\right),dr^2, \tau\right) +\frac{n}{2}\log\frac{\tau}{r^2}.
\end{align*}
\begin{proof}
The expression is the following :
\begin{align*}
\int_N \int_0^\infty(\partial_r (a_r+\tilde{f}))^2&\left(\frac{e^{-a_r}}{(4\pi\tau)^{\frac{1}{2}}} dr\right)\left(\frac{e^{-\tilde{f}}}{\left(4\pi\tau r^{-2}\right)^{\frac{n}{2}}} dv\right) \\
&= \int_0^\infty\left(\int_N (\partial_r (a_r+\tilde{f}))^2\left(\frac{e^{-\tilde{f}}}{\left(4\pi\tau r^{-2}\right)^{\frac{n}{2}}} dv\right)\right)\left(\frac{e^{-a_r}}{(4\pi\tau)^{\frac{1}{2}}} dr\right)\\
&\geqslant \int_0^\infty \left(\int_N \partial_r (a_r+\tilde{f})\left(\frac{e^{-\tilde{f}}}{\left(4\pi\tau r^{-2}\right)^{\frac{n}{2}}} dv\right)\right)^2 \left(\frac{e^{-a_r}}{(4\pi\tau)^{\frac{1}{2}}} dr\right) \tag*{By Jensen inequality.} \\
&\geqslant\int_0^\infty \left(\partial_ra_r+\int_N (\partial_r \tilde{f})\left(\frac{e^{-\tilde{f}}}{\left(4\pi\tau r^{-2}\right)^{\frac{n}{2}}} dv\right)\right)^2 \left(\frac{e^{-a_r}}{(4\pi\tau)^{\frac{1}{2}}} dr\right) \\
&= \int_0^\infty \left(\partial_r a_r + \int_N (\partial_r \tilde{f})(4\pi\tau r^{-2})^{-\frac{n}{2}} e^{-\tilde{f}}dv\right)^2 \left(\frac{e^{-a_r}}{(4\pi\tau)^{\frac{1}{2}}} dr\right).
\end{align*}
Let us see what we can do with the term :
$$\int_N (4\pi\tau r^{-2})^{-\frac{n}{2}} e^{-\tilde{f}}(\partial_r \tilde{f})dv.$$
\begin{align*}
\int_N (4\pi\tau r^{-2})^{-\frac{n}{2}} e^{-\tilde{f}}(\partial_r \tilde{f})dv &= (4\pi\tau)^{\frac{n}{2}}\int_N r^n e^{-\tilde{f}}(\partial_r \tilde{f})dv \\
&=-(4\pi\tau)^{-\frac{n}{2}}\int_N r^n \partial_r\left(e^{-\tilde{f}}\right)dv \\
&=-\left[(4\pi\tau)^{-\frac{n}{2}}\int_N \partial_r\left(r^ne^{-\tilde{f}}\right)dv - \frac{n}{r}\right]\tag*{Because $\int_N (4\pi\tau r^{-2})^{\frac{n}{2}} e^{-\tilde{f}} = 1$.}\\
&=-\left[(4\pi\tau)^{-\frac{n}{2}}\partial_r\left(\int_N r^ne^{-\tilde{f}}dv\right) - \frac{n}{r}\right] \\
&= \frac{n}{r}  \tag*{Because $\int_N r^ne^{-\tilde{f}}dv$ is constant.}\\
&= \partial_r \left(\frac{n}{2}\log\frac{\tau}{r^2}\right).
\end{align*}
We finally get :
\begin{align*}
\int_N \int_0^\infty (\partial_r (a_r+\tilde{f}))^2\left(\frac{e^{-a_r}}{(4\pi\tau)^{\frac{1}{2}}} dr\right)\left(\frac{e^{-\tilde{f}}}{\left(4\pi\tau r^{-2}\right)^{\frac{n}{2}}} dv\right) \geqslant \int_0^\infty\left(\partial_r \left(a_r-\frac{n}{2}\log\frac{\tau}{r^2}\right)\right)^2 \left(\frac{e^{-a_r}}{(4\pi\tau)^{\frac{1}{2}}} dr\right).
\end{align*}
Plugging this inequality in the expression of the lemma, we get :
\begin{align*}
\int_N \int_0^\infty [\tau(\partial_r( a_r+\tilde{f}))^2+a_r-1]&\left(\frac{e^{-a_r}}{(4\pi\tau)^{\frac{1}{2}}} dr\right)\left(\frac{e^{-\tilde{f}}}{\left(4\pi\tau r^{-2}\right)^{\frac{n}{2}}} dv\right)\\
&\geqslant \int_0^\infty[\tau(\partial_r(a_r+n\log r))^2+a_r-1]\left(\frac{e^{-a_r}}{(4\pi\tau)^{\frac{1}{2}}} dr\right).
\end{align*}
\end{proof}
\end{lem}
\paragraph{2) Let us now use the lemma \eqref{largetau} to get a lower bound on $\mathcal{W}^{C(N)}(f,g,\tau)$} expressed as a functional on $\mathbb{R}^+$ namely :
\begin{lem}\label{lowbdmu}
There exists $A>0$ such that :
\begin{align*}
\mathcal{W}^{N}(\tilde{f},g,\tau r^{-2})-n(n-1)\frac{\tau}{r^2}\geqslant \left( \lambda^N-n(n-1) \right)\frac{\tau}{r^2}-A-\frac{n}{2}\log_+ \left(\frac{C^2\tau}{r^2}\right)
\end{align*}
\begin{proof}
It is a direct consequence of \eqref{largetau}.
\end{proof}
\end{lem}
\paragraph{We have : $\left( \lambda^N-n(n-1) \right)>-(n-1)^2$} since by assumption, $\lambda^N>(n-1)$
\\
As a consequence of Lemma \ref{lowbdmu}, it is enough to bound from below the following simpler quantity  :
\begin{lem}
Let us note $K:=\left( \lambda^N-n(n-1)\right)>-(n-1)^2$, we have by a direct rephrasing with the new notations :
\begin{align*}
\mathcal{W}^{C(N)}(f) \geqslant&\int_0^{C\sqrt{\tau}} \left[K\frac{\tau}{r^2} +\tau\left(\partial_r \left(a_r-\frac{n}{2}\log \left(4\pi\frac{\tau}{r^2}\right)\right)\right)^2+ a_r-\frac{n}{2}\log \left(4\pi\frac{\tau}{r^2}\right)\right] \left(\frac{e^{-a_r}}{(4\pi\tau)^{\frac{1}{2}}} dr\right)\\
&-A-1-\frac{n}{2}\log\left(\frac{C^2}{4\pi}\right).
\end{align*}

\end{lem}
\indent Noting $w = \frac{e^{-a_r}}{(4\pi\tau) r^n}$, we can rewrite this as :
\begin{align*}
\mathcal{W}^{C(N)}(f,\tau)\geqslant & \int_0^{C\sqrt{\tau}} \left[\tau\left(4(\partial_r (w))^2+\frac{K}{r^2}w^2 \right)+w^2\log(w^2)\right]r^ndr-\frac{n+1}{2}\log(4\pi\tau)\\
&-(A+1)-\frac{n}{2}\log\left(\frac{C^2}{4\pi}\right).
\end{align*}
The goal of the next pages is to bound the term :
$$\int_0^{C\sqrt{\tau}} \left[\tau\left(4(\partial_r w)^2+\frac{K}{r^2}w^2 \right)+w^2\log(w^2)\right]r^ndr-\frac{n+1}{2}\log\tau,$$ from below for smooth $w$ such that 
$$\int_0^{C\sqrt{\tau}}w^2 r^ndr \leqslant 1.$$

\paragraph{3) Let us work on $[0,\infty]$ for easier notations and bound}
$$\int_0^\infty \left[\tau\left(4(\partial_r (w))^2+\frac{K}{r^2}w^2 \right)+w^2\log(w^2)\right]r^ndr-\frac{n+1}{2}\log\tau,$$
for $w$ such that : 
$$\int_0^{\infty}w^2 r^ndr = 1.$$
\begin{rem}
This is stronger than what we need to finish the proof as, given a $w$ such that : $0<\int_0^{C\sqrt{\tau}}w^2 r^ndr \leqslant 1$, we can consider $\tilde{w} = \frac{w}{\sqrt{\int_0^{C\sqrt{\tau}}w^2 r^ndr}}$ in $B(0,C\sqrt{\tau})$ and $\tilde{w} = 0$ outside.
\end{rem}
\indent The idea of this estimate is to use the $|\partial_r w|^2$ term to first control the $u^2\log u^2$ term by a weighted log-Sobolev inequality and then use the remaining $|\partial_r w|^2$ term to use a weighted Hardy inequality.

\subparagraph{a. The $w^2\log w^2$ term - weighted log-Sobolev inequality}
  :

 The sharp weighted log-Sobolev inequality on $\mathbb{R}^+$ we use comes directly from the Euclidean case inequality with  radial functions :
\begin{lem}\label{log sobolev radial}
 We have the following log-Sobolev inequality whose sharpness comes from the Euclidean case of the unit sphere :

 For all $\tau_0>0$ :
\begin{align}
4\tau_0\int_0^\infty r^n |w'|^2dr \geqslant\int_0^\infty r^n w^2\log(w^2) dr+\frac{n+1}{2}\log(4\pi\tau_0)+(n+1)-\log(vol(\mathbb{S}^n)), \label{logsob}
\end{align}
for $\int_0^\infty r^n w^2 = 1$.
\begin{proof}
The idea of the proof is to apply the classical log-Sobolev inequality on the Euclidean space to radial functions.
\\
\indent By the classical log-Sobolev inequality, the $\nu$-functional of the Euclidean space $\mathbb{R}^{n+1}$ is 0.
\\
\\
\indent By considering a radial function $v: \mathbb{R}^n\to \mathbb{R}$,

 $$v(r,.)=\frac{w(r)}{\sqrt{vol(\mathbb{S}^n)}}$$ such that $\int_{\mathbb{R}^{n+1}} v^2 =1$, that is : $\int_{\mathbb{R}^+}r^n w(r)^2 = 1$ and rewriting the fact that :
$$\mathcal{W}^{\mathbb{R}^{n+1}}(f,\tau,g_{eucl})\geqslant 0,$$
where $f$ is such that : $v^2 = \frac{e^{-f}}{(4\pi\tau)^{\frac{n+1}{2}}}$.
\\
\indent We get :
\\
\indent For all $w: \mathbb{R}^+\to \mathbb{R}^+$ such that $\int_{\mathbb{R}^+} r^nw^2 = 1$, we have :
$$\int_{\mathbb{R}^+} r^n [4\tau(\partial_r w)^2-w^2\log w^2]dr-\frac{n+1}{2}\log(4\pi\tau) -(n+1)+\log(vol(\mathbb{S}^n))\geqslant 0.$$
\end{proof}
\end{lem}
\indent This gives a sharp inequality \textit{for all $\tau_0$}. 
\\
\indent Let us choose one $\tau_0$ to define later with which we will use this inequality. 
\begin{rem}
$\tau_0$ will be chosen as the optimal constant to bound the remaining term thanks to a weighted Hardy inequality.
\end{rem}
\subparagraph{b. The $\frac{K}{r^2}w^2$ term - Hardy inequality}
:
\\
\indent Now, let us see what freedom we have on $\tau_0$ to apply the weighted Hardy inequality to the remaining term : 
Let us see until which value of $\tau\geqslant\tau_0>0$, the functional
\begin{align}
\int_{\mathbb{R}^+}r^n\int_N\left(4(\tau-\tau_0)|\partial_r w|^2+\tau\frac{K}{r^2}w^2\right)dvdr. \label{postlogsob}
\end{align}
 is nonnegative, for all smooth $w$.

\begin{rem}

If $K \geqslant 0$, then it is always the case, and we will take $\tau_0 = \tau$.
\\

 Let us assume that $K<0$.
\end{rem}
From \cite{hardy}, we have the following weighted Hardy inequality on $\mathbb{R}^+$ :
\begin{thm}
\indent For all $v$ such that : 
$$0<\int_0^{+\infty}r^{n-2}v^2<+\infty.$$
\indent We have the following weighted Hardy inequality :
\begin{align}
\int_0^{+\infty}r^{n-2}v^2dr<\frac{4}{(n-1)^2}\int_0^{+\infty}r^n(v')^2dr.\label{hardy}
\end{align}
\end{thm}
We have assumed that $\lambda^N>(n-1)$ which implies : $K >-(n-1)^2$, so we have the following lemma as a consequence of the weighted Hardy inequality that gives a condition on the combination $(\tau,\tau_0)$ that we will use in the following.
\begin{lem}
\indent If the pair $(\tau,\tau_0)$ satisfies : $\frac{4(\tau-\tau_0)}{-\tau K} \geqslant \frac{4}{(n-1)^2}$ that is : 
\begin{align}
\tau\geqslant\frac{\tau_0}{\left(1-\frac{-K}{(n-1)^2}\right)}.\label{condtautau0}
\end{align}
\indent We have, by \eqref{hardy} :
\begin{align*}
\int_{\mathbb{R}^+}r^n\int_N\left(4(\tau-\tau_0)|\partial_r w|^2+\tau\frac{K}{r^2}w^2\right)dvdr &> \int_{\mathbb{R}^+}r^{n-2} (\tau K+(\tau - \tau_0)(n-1)^2)w^2dr\\
&>0.
\end{align*}
\end{lem}
\subparagraph{c. Let us choose such a $(\tau,\tau_0)$ combination for the following}
:
\\
\indent Let us rewrite the quantity we want to bound from below to emphasize how we use each inequality :
\begin{align}
\int_0^\infty& r^n\int_N\left[\tau\left(4|\partial_r w|^2+\frac{K}{r^2}w^2\right)-w^2\log w^2\right]dvdr-\frac{n+1}{2}\log(\tau)-(n+1)\nonumber\\
&=\int_0^\infty r^n\int_N\left[(\tau-\tau_0)4|\partial_r w|^2+\tau\frac{K}{r^2}w^2\right]dvdr \label{har term}\\
&+\int_{\mathbb{R}^+}  r^n(4\tau_0 |\partial_r w|^2dr -w^2\log(w^2))dr-\frac{n+1}{2}\log(\tau_0)-(n+1)+\log(vol(\mathbb{S}^n))\label{logsob term}\\
&-\frac{n+1}{2}\log\left(\frac{\tau}{\tau_0}\right)-\log(vol(\mathbb{S}^n)).\label{remaining}
\end{align}
\indent Here, thanks to \eqref{hardy} and \eqref{logsob}, \eqref{har term} and
\eqref{logsob term} are nonnegative and we get :
\begin{align*}
\mathcal{W}^{C(N)}(u,\tau)&\geqslant -\frac{n+1}{2}\log\left(\frac{\tau}{\tau_0}\right)-\log(vol(\mathbb{S}^n))\\
&\geqslant \frac{n+1}{2}\log\left(1-\frac{-K}{(n-1)^2}\right)-\log(vol(\mathbb{S}^n))>-\infty,
\end{align*}
where the last step is done by choosing the the smallest $\tau$ possible at a given $\tau_0$ (or biggest $\tau_0$ at a given $\tau$) to satisfy \eqref{condtautau0}.
\\

 Since the rest of the expression of $\mathcal{W}^{C(N)}(f,g,\tau)$ is finite by the lemma \ref{control a grande distance}, we have the result :
 There exists $D = 1+A+\frac{n+1}{2}\log\left(1-\frac{-K}{(n-1)^2}\right)-\log(vol(\mathbb{S}^n))$ such that :
$$\nu^{C(N)}>-D>-\infty.$$

\end{proof}
\begin{rem}
Looking back at the proof, we can get the following lower bound on the $\nu$-functional of cones over some perturbations of the unit sphere :
\end{rem}
\begin{cor}
Let us consider a manifold $N$ such that there exists $\epsilon_1>0$, $\epsilon_2>0$ and $\epsilon_3>0$ small enough such that its $\mathcal{W}$-functional satisfies the lower bound noted $L(\epsilon_1,\epsilon_2,\epsilon_3)$ defined by :

For all $f : N\mapsto \mathbb{R}$ and $\tau > 0$
\begin{align*}
\mathcal{W}^{N}(f,\tau)\geqslant \tau(1-\epsilon_1)\mathcal{F}^{\mathbb{S}^n}\left(f+\delta,g^{\mathbb{S}^n}\right)+(1+\epsilon_2)\mathcal{N}^{\mathbb{S}^n}\left(f+\delta,g^{\mathbb{S}^n}\right)+\frac{n}{2}\log{4\pi\tau}-n-\epsilon_3,
\end{align*}
where $\delta$ is defined to ensure that $\int_{\mathbb{S}^n}\frac{e^{-f-\delta}}{(4\pi\tau)^{\frac{n}{2}}}=1$.

Then, $$\nu^{C(N)}\geqslant -\Psi(\epsilon_1,\epsilon_2,\epsilon_3),$$ 
where $\Psi>0$ and $\Psi(\epsilon_1,\epsilon_2,\epsilon_3)\to 0$ when $(\epsilon_1,\epsilon_2,\epsilon_3)\to (0,0,0)$.
\begin{rem}
The case $\epsilon_1 =\epsilon_2 = \epsilon_3 =0$ corresponds to the $\mathcal{W}$-functional of the sphere. That is why we will consider such manifolds as perturbations of the sphere from the point of view of the $\mathcal{W}$-functional.

 Let us also note that such a bound implies that $\lambda^N\geqslant (1-\epsilon_1)n(n-1)$ which is higher than $(n-1)$ for small $\epsilon_1$.
\end{rem}
\end{cor}
\begin{proof}
The proof is basically the same as the last proposition, but uses the log-Sobolev inequality a bit differently.
\\
\indent By the separation of variables formula, the lower bound $L(\epsilon_1,\epsilon_2,\epsilon_3)$ on the link $N$ implies the following lower bound on the cone $C(N)$ :
\begin{align*}
\mathcal{W}^{C(N)}(f,g,\tau)&\geqslant \int_0^\infty \left[\frac{\tau}{r^2}\left((1-\epsilon_1)\mathcal{F}^{\mathbb{S}^n}-n(n-1)\right)+(1+\epsilon_2)\mathcal{N}^{\mathbb{S}^n}-\frac{n}{2}\log\left(4\pi\frac{\tau}{r^2}\right)-n\right]\left(\frac{e^{-a_r}}{(4\pi\tau)^{\frac{1}{2}}} dr\right)\\
&-\delta-\epsilon_3 \\
&+\int_{\mathbb{S}^n} \left(\int_0^\infty \left[\tau(\partial_r f)^2+a_r-1\right]\left(\frac{e^{-a_r}}{(4\pi\tau)^{\frac{1}{2}}} \right)\right)\left(\frac{e^{-\tilde{f}}}{\left(4\pi\tau r^{-2}\right)^{\frac{n}{2}}}\right)drdv,
\end{align*}
where the functionals on $\mathbb{S}^n$ are taken at $f+\delta+\frac{n}{2}\log(4\pi\tau)$
\\

 The next idea is to substract 
$$(1+\epsilon_2)\mathcal{W}^{C(\mathbb{S}^n)}(f+\delta,g,\tau_0)\geqslant 0$$ 
from the previous lower bound.
\\
\indent After using Jensen inequality to take care of the $-\epsilon_2 a_re^{-a_r}$ term that is left, the expression left is :
\begin{align*}
\mathcal{W}^{C(N)}(f,g,\tau)&\geqslant \int_0^\infty \left[\frac{\tau}{r^2}\left((1-\epsilon_1)\mathcal{F}^{\mathbb{S}^n}-n(n-1)\right)-(1+\epsilon_2)\frac{\tau_0}{r^2}\left(\mathcal{F}^{\mathbb{S}^n}-n(n-1)\right)\right]\left(\frac{e^{-a_r}}{(4\pi\tau)^{\frac{1}{2}}} dr\right)\\
&\frac{n+1}{2}\log\left(\frac{\tau}{\tau_0^{1+\epsilon^2}}\right)-\epsilon_2n-\delta-\epsilon_3 \\
&+\int_{\mathbb{S}^n} \left(\int_0^\infty \left[(\tau-(1+\epsilon_2)\tau_0)(\partial_r f)^2-\epsilon_2\right]\left(\frac{e^{-a_r}}{(4\pi\tau)^{\frac{1}{2}}} \right)\right)\left(\frac{e^{-\tilde{f}}}{\left(4\pi\tau r^{-2}\right)^{\frac{n}{2}}}\right)drdv.
\end{align*}
An application of the weighted Hardy inequality like in the theorem implies a lower bound that tend to $0$ as $(\epsilon_1,\epsilon_2,\epsilon_3)$ tends to $(0,0,0)$.

\end{proof}

\subsection{Perelman's functionals on manifolds with conical singularities and asymptotically conical manifolds}

Let us now give a first application of the previous finiteness results to the study of manifolds asymptotically conical and manifold with conical singularities thanks to Perelman's functionals.

\begin{defn}\label{manifold with conical singularities}
	We will say that a metric space $(M,g)$ is a \emph{manifold with conical singularities} modeled on the cones $C(N_1),...,C(N_k)$ at the points $x_1,...,x_k$ if $(M\backslash\{x_1,...,x_k\},g)$ is a smooth manifold and if, for each $j \in \{1,...,k\}$, there exist $\epsilon_j>0$ and, a diffeomorphism $\phi_j : (0,\epsilon_j)\times N_j\to B(x_j,\epsilon)$, such that as $r\to 0$, for all $k\in\mathbb{N}$ :
	$$r^k|\nabla^k\left(\phi_j^*(g)-\left(dr^2 + r^2g^{N_j}\right)\right)|\to 0.$$
\end{defn}
\begin{defn}
	We will say that a complete manifold $(M,g)$ is \textit{smoothly asymptotic to the cone $C(N)$} at infinity if there exists a compact $K\subset M$, a real $R>0$ and a diffeomorphism 
	$$\phi : (R,+\infty)\times N\to M\backslash K$$ such that, as $r\to +\infty$, for $k\in \mathbb{N}$, 
	
	$$r^k|\nabla^k\left(\phi^*(g)-\left(dr^2 + r^2g^{N}\right)\right)|\to 0.$$
\end{defn}
\begin{note}
	The norms and derivatives are computed thanks to the cone metric.
\end{note}

\begin{cor}
	Let $(M^n,g)$ be a compact manifold with conical singularities such that one singularity is modeled on a cone $C(N)$ on a section $N$ such that $\lambda^N < (n-2)$.
	
	Then $$\lambda^M = -\infty.$$
	
	Conversely, if each singularity model is modeled on a cone $C(N_i)$ with a link $N_i$ such that $\lambda^{N_i} > (n-2)$, then, $$\lambda^M > -\infty.$$
\end{cor}

\begin{proof}
		
	By the definition of a conical singularity, in small balls around a conical singularity, there exists coordinates in which the manifold is a warped product of the form of the ones studied in the appendix A.3. In particular, we can control how far the expression of $\mathcal{F}^M$ differs from $\mathcal{F}^N$ in the same coordinates (see appendix A.3.). As a consequence, for all conical singularity, there exists $\epsilon'_i>0$ small enough such that the difference between the expressions is less than $|\lambda^{N_i} - (n-2)|$, so, in the two cases : 
	
	\begin{enumerate}
		\item If $\lambda^{N_i}<(n-2)$ for a conical singularity, then it is infinite.
		
		Like in the proof of the cone case, we can consider a function $v = \frac{b}{r^a}1_{[0,r_0]}$ for $r_0<\epsilon'$ small enough and estimate just like in 2.3.1.
		
		\item If $\lambda^{N_i} > (n-2)$ for all conical singularities, then it is finite. It is a little more complicated to see :
				
		On each of the balls $B(x_i,\epsilon_i)$, we can use the weighted Hardy inequality of \cite{hardy} just like in 2.3.2. and get a lower bound on the $\lambda$-functional on the balls $B(x_i,\epsilon'_i)$. The rest of the manifold being smooth and compact, there exists another lower bound for the $\mathcal{F}$-functional.
	\end{enumerate}
\end{proof}

\begin{rem}
	We cannot decide if $\lambda^M$ is finite or not for a conical singularity modeled on $C(N)$ such that $\lambda^N = (n-2)$. 
	
	Ideed, considering that the metric becomes conical at a rate $\epsilon(r)$ and by using a function $v = \frac{b}{r^{\frac{n-1}{2}}}1_{[0,r_0]}$, we get that $$\lambda^M \geqslant\int_0^{\infty}\frac{\mathcal{O}(\epsilon(r))}{r}dr,$$ and if $\epsilon(r)$ does not converge to zero fast enough and the $\mathcal{O}(\epsilon(r))$ happens to be negative (this is easy to construct as a warped product on $(\mathbb{R}\times \mathbb{S}^{n-1},dr^2 + h(r)^2 g^{\mathbb{S}^{n-1}})$), this may diverge.
	\\
	
	 If we assume a convergence speed in a positive power of $r$ for example, then, $\lambda^N = (n-2)$ also implies that $\lambda^M >-\infty$ by looking more closely at the proof of the \emph{strict} Hardy inequality in \cite{hardy}.
\end{rem}

\begin{cor}
	For $(M^n,g)$, a manifold smoothly asymptotic to the cone $C(N)$ at infinity.
	
	If $\lambda^N\leqslant (n-2)$, then, $$\nu^M(g) = - \infty.$$
	
	If, $\lambda^N > (n-2)$, then, $$\nu^M \leqslant \nu^{C(N)}.$$
\end{cor}

\begin{rem}
	The result is not a characterization of the finiteness of the functional as it was for manifolds with conical singularities, but the $\nu$-functional of the cones at infinity give an upper bound for the entropy of the whole manifold. A similar upper bound for manifolds with conical singularities is true for the $\mu$-functional.
\end{rem}

\section*{Acknowledgements}
	I would like to thank Richard Bamler for inviting me at UC Berkeley and supervising this work.

\section{A global Pseudolocality result}
Perelman's Pseudolocality theorem states that if a region $\Omega$ is "close enough" to the Euclidean space, then there is a smaller region inside, which has bounded curvature for a small time. 

 In particular this region will be nonsingular for small times.
\\

 Looking at the proof of this theorem, we realize that the "close enough" to the Euclidean space hypothesis is only used to ensure that for small $\tau$, the $\mu(g,\tau)$-functional of $\Omega$ is larger than $-\eta_n$ for some $\eta_n>0$.
\\

 Our modification states that there exists $\eta_n>0$ such that if for all $\tau \leqslant T_0$, $\mu(g,\tau)\geqslant-\eta_n$, then the flow is nonsingular on $[0,T_0]$ and has a bound on its curvature tensor on this interval only depending on the values of the $\mu(g,\tau)$ for $\tau\in [0,T_0]$. 

 In dimension 3, from the description of the singularities given by Perelman, we have an explicit value : 
$$\eta_3 = 1-\log 2 = \nu(\mathbb{S}^2).$$
\subsection{The classical theorem} 
Let us state the theorem, the proof can be found in Kleiner-Lott notes \cite{KleinLott}.
\begin{thm}[cf. Theorem I.10.1, Perelman] For every $\alpha > 0$ there exist $\delta$, $\epsilon > 0$ with the following property :
\\

 Suppose that we have a smooth pointed Ricci flow solution $(M,(x_0, 0), g(.))$ defined for $t \in [0,(\epsilon r_0)^2]$, such that each time slice is complete.

 Suppose that for any $x \in B_0(x_0, r_0)$ and $\Omega \subset B_0(x_0, r_0)$, we have \\
$$\R(x,0) \geqslant -r_0^{-2}$$ and $$vol(\partial\Omega)^n \geqslant (1 - \delta) c_n vol(\Omega)^{n-1},$$ where $c_n$ is the Euclidean isoperimetric constant. 
\\

Then, $|\Rm|(x,t) < \alpha t^{-1} + (\epsilon r_0)^{-2}$ whenever $0 < t \leqslant (\epsilon r_0)^2$
and $d(x, t) = dist_t(x, x_0) \leqslant \epsilon r_0$.
\end{thm}
\begin{proof}
The proof is a proof by contradiction, by considering a sequence of pointed manifolds that are counterexamples of the property : $(M_k, x_{0,k},g_k(.))$ on an interval $[0,\tau_k]$. The idea is to first, find an upper bound on the $\mu$-functional of these manifolds for $k$ large enough, and then contradict this bound thanks to a bound on the isoperimetric constant and the scalar curvature that imply a lower bound on $\mu(\tau)$ for small $\tau$.
\\

The crucial element to note is that the lower bound on the isoperimetric constant and the scalar curvature are only used at the very last step of the proof to contradict an upper bound for $\min_{0\leqslant\tau\leqslant\tau_k}\mu^{M_k}(g_k,\tau)$
given by the lemma 33.4 (thanks to conjugate heat kernel centered at a well chosen high curvature space time point).
\\

 In particular, any condition that contradicts this upper bound of that $\min_{\tau\leqslant\tau_k}\mu(g_k,\tau)$ would imply the result.
\end{proof}

\subsection{A global version}
The only step where the isoperimetric constant of our domain is used to have a bound on the curvature is the very last step.

 In particular, manifolds that have a $\nu$-functional (the minimum of the $\mu(\tau)$) close enough to $0$ contradict the lemma 33.4 and we can rephrase the result as :
\begin{prop}[Global pseudolocality]\label{global pseudolocality}
	For all $\alpha>0$, there exists a constant $\eta_n(\alpha)>0$ such that :
	
	If $(M,g)$ is a complete Riemannian manifold with bounded curvature such that there exists $T > 0$ such that, for all $0 < \tau \leqslant T$,
	
	$$\mu(g,\tau) > - \eta_n(\alpha),$$
	then, the flow exists for all $0<t<T$ and :
	
	$$|Rm(g_t)|\leqslant \frac{\alpha}{t}.$$
\end{prop}

\begin{proof}
	The proof of the statement is adaptation of the classical pseudolocality result, let us give a description of it :
	
	\begin{enumerate}
		\item The first step is to choose a good point that we will be able to blow up at the scale of the curvature :
		
		\begin{lem}[Point selection, Lemma 31.1 \cite{KleinLott}]
			For all $A>0$, there exists $\epsilon>\frac{1}{100 n A}$ such that the solution to the Ricci flow exists on $[0,\epsilon^2]$ and there exists a point $(\bar{x},\bar{t})$ such that :
			
			For all $(x,t)\in B_{\bar{t}}\left(\bar{x}, \frac{A}{10} |Rm(\bar{x},\bar{t})|^{-\frac{1}{2}}\right)\times \left[\bar{t}- \frac{1}{2|Rm(\bar{x},\bar{t})|}\right]$,
			$$|Rm(x,t)|\leqslant 4|Rm(\bar{x},\bar{t})|.$$
		\end{lem}
		
		\item By parabolically scaling the previous Ricci flow by its curvature at $(\bar{x},\bar{t})$, we are in the case of a manifold satisfying :

		$$\sup_M|Rm_T| = |Rm_T(\bar{x})| = 1 > \frac{\alpha}{T}$$
		for some $T > \alpha > 0$, and such that for all $(x,t)\in B_T\left(\bar{x},\frac{2}{100n}\right) \times [0,T]$, $$|Rm_t(x)|\leqslant 4.$$

		\item Let us prove that there is a strictly negative upper bound to the $\mu$-functional. 
		
		\begin{rem}
			This bound will be the $-\eta$ we are looking for and adding the assumption that $\mu(\tau,g) > -\eta$ would lead to a contradiction.
		\end{rem}
		
		\begin{lem}[Adaptation of the proof of Lemma 33.4, \cite{KleinLott}]
			For any Ricci flow $(M^n,g(t))_{0\leqslant t\leqslant T}$ such that 
			$$\sup_M|Rm_T| = |Rm_T(\bar{x})| = 1 > \frac{\alpha}{T}$$
			for some constants $T > \alpha > 0$, and such that for all $(x,t)\in B_T\left(\bar{x},\frac{2}{100n}\right) \times [0,T]$, $$|Rm_t(x)|\leqslant 4,$$ 
			there exists $\tilde{t}\in [T-\frac{\alpha}{2},T]$ and $\eta(n,\alpha) > 0$ and such that $\mu(g(\tilde{t}),T-\tilde{t}) \leqslant -\eta(n,\alpha)$. 
		\end{lem}
		
		\begin{proof}
			Let us prove this result by contradiction. Let us fix $\alpha>0$ and consider a sequence of counterexamples, that is, a sequence of solutions to the Ricci flow $(M_k,g_k(t))_{0\leqslant t\leqslant T_k}$ pointed at $(\bar{x}_k,T_k)$ and $T_k > \alpha$ such that for all $T_k-\frac{\alpha}{2}<t_k<T_k$ :
			
			$$\liminf_{k\to\infty}\sup_{0<t_k<T_k}\mu(g_k(t_k),T_k-t_k) \geqslant 0.$$
			
			Moreover, let us consider the fundamental solutions $u_k$ to the backward heat equation starting at a Dirac at $\bar{x}_k$ at the time $T_k$, that is :
			\begin{align*}
			\partial_t u_k = - \Delta u_k + R_ku_k
			\end{align*}
			and the limit at $T_k$ is the Dirac at $\bar{x}_k$.
			\\

			There are two cases :
			\begin{enumerate}
				\item There is a lower bound on the injectivity radius implying that it is possible to take a sublimit of the sequence by Hamilton's compactness theorem. Let us note it $(M_\infty,g_\infty(.))$ it is defined on a time interval $[T_\infty-\frac{\alpha}{2},T_\infty]$.
				\\
				
				It satisfies $|Rm_\infty|\leqslant 4$ in $B(\bar{x}_\infty, \frac{2}{100n})\times [0,T_\infty]$ and $|Rm_\infty(\bar{x}_\infty,T_\infty)| = 1$.
				
				thanks to the lemma 33.1 of \cite{KleinLott}, the fundamental solution to the backwards heat equation starting at $\bar{x}_\infty,T_\infty$, $u_\infty$ is the smooth limit (on compact subsets, and in particular in $B(\bar{x}_\infty,\frac{1}{100n}))$.
				
				Now, given the smooth convergence, we also have the convergence of 
				
				$v_k(t) := \left[\left(T_k-t\right)(2\Delta_k f_k-|\nabla_k f_k|^2+\R_k)-f_k-n\right]u_k$ associated to $u_k$ to $v_\infty$ associated to $u_\infty$ (we used the usual notation $u = \frac{1}{\left(4\pi(T-t)\right)^{\frac{n}{2}}}e^{-f}$ to define $f$). By Perelman's Harnack inequality, $$v_\infty\leqslant 0.$$ 
				
				Let $\tilde{t}_\infty\in (T_\infty-\frac{\alpha}{2},T_\infty)$, and let us consider $h$ a solution to the forward heat equation starting at a nonnegative function exactly supported in $B_\infty :=B_{\tilde{t}_\infty}(\bar{x}_\infty,\sqrt{T_\infty-\tilde{t}_\infty})$ at $\tilde{t}_\infty$. By Perelman's Harnack inequality :
				
				$$\frac{d}{dt}\int_{B_\infty}hv_\infty dv_\infty \geqslant 0.$$
				
				Moreover, it tends to $0$ as $t$ tends to $T_\infty$. This implies that for all $t\in (\tilde{t}_\infty,T_\infty)$, $\int_{B_\infty}hv_\infty dv_\infty = 0$ because $v_\infty\leqslant 0$. Now since $h$ is strictly positive for $t>\tilde{t}_\infty$, for all $t\in (\tilde{t}_\infty,T_\infty)$, $$v_\infty = 0.$$
				
				This implies that the $\mathcal{W}$-functional at $f_\infty$ (which is the integral of $v_\infty$) is constant when $(g_\infty,f_\infty)$ satisfy the equations of evolution \eqref{evolution equation W}.
				
				This means that they satisfy the following shrinking soliton equation :
				
				$$\Ric_\infty + \nabla^2_\infty f_\infty = \frac{1}{2(T_\infty-t)}.$$
				
				But as we have seen, this means that up to diffeomorphism, the Ricci flow at $(M_\infty,g_\infty)$ just acts by scaling by $T_\infty-t$. In particular, the (non vanishing) curvature blows up at rate $\frac{1}{T_\infty-t}$ when $t\to T_\infty$ which is a contradiction to our curvature bounds.
				
				\item In the second case, it is possible to blow up at the scale of the collapsing injectivity radii to ensure that the injectivity radius equals $1$ to end up with a smooth limiting Ricci flow. 
				
				Because the injectivity radius was collapsing at the scale of the curvature, we end up with a flat Ricci flow which is not the Euclidean space because its injectivity radius is $1$. Since we have blown up even more than the scale of the curvature, the range of $\tau$ to consider for the $\mathcal{W}$ functional is arbitrarily large and in particular, like in the proof of lemma 33.4 in \cite{KleinLott}, we can see that the $\mathcal{W}$ associated to a good cut-off of the backwards heat equation $f_\infty$ leads to :
				
				$$\lim_{\tau\to\infty}\mathcal{W}(f_\infty,g_\infty,\tau) = -\infty.$$
				
				and this implies that there exists $\tilde{t}_k$ depending on the scale of the injectivity radius for each $k$, such that :
				
				$$\lim_{k \to\infty}\mathcal{W}(f_k,g_k(\tilde{t}_k),T_k-\tilde{t}_k) = -\infty.$$
				which is also a contradiction.
				
				\begin{rem}
					Another way to see it is to realize that there is collapsing and that collapsing together with a lower bound on the scalar curvature gives an arbitrarily negative $\nu$-functional (see the non-collapsing theorems in \cite{KleinLott}).
				\end{rem}
				
			\end{enumerate}
		\end{proof}
		
		\item This implies the statement of the proposition, indeed, 
		$$t\mapsto \mu\left(g_k(t),\tau-t\right)  $$
		is increasing and in particular, if for all $0<\tau<T$, 
		$$\mu\left(g_k(0),\tau\right)>-\eta_n(\alpha),$$
		then by monotonicity, we have for all $\tau \geqslant t\geqslant 0$ :
		$$\mu\left(g_k(t),\tau-t\right)>-\eta_n(\alpha).$$
		
		In particular, by the previous lemmas, every $t_k$ when the control of the curvature is not satisfied has to be larger than $T$. Which is the statement of the lemma.
	\end{enumerate}
\end{proof}

\begin{cor}
	With the same assumptions, if we assume that $\nu(g)$ is larger then the infimum of the $\eta_n(\alpha)$ for $\alpha > 0$, then the Ricci flow starting at $(M,g)$ is a type III solution existing for all times.
\end{cor}

\begin{rem}
	This result using Perelman's functionals instead of bounds on the scalar curvature and isoperimetric constants gives control for potentially large times and ball radii. It is also compatible to our study of Perelman's functionals on cones as this allows us to take advantage of the scale invariance of cones to construct type III solutions of the Ricci flow coming out of some cones in the next section.
\end{rem}

\begin{rem}
In dimension $3$, thanks to the description of the possible singularities of the Ricci flow given by Perelman's canonical neighborhood theorem, we can find the constant $\eta_3$ explicitely : $\eta_3 = \nu^{\mathbb{R}\times \mathbb{S}^2} = 1-\log 2$ (with equality in the case of the formation of necks).

\end{rem}

\section{Constructing nonsingular type III flows coming out of some cones}
In this section, we will present an application of the global pseudolocality result from the previous section to the construction of nonsingular Ricci flows coming out of some cones. The "some cones" will correspond to cones over particular perturbations of the unit sphere.
\\
\indent The construction is done in a few steps :
\begin{itemize}
\item Find a condition on the link $N$ under which it is possible to smooth out the cone $C(N)$ into a manifold $M$ that has a high enough $\nu$-functional to apply the global pseudolocality. 

\item Thanks to the global pseudolocality result, there exists a Ricci flow starting at $(\mathbb{R}^{n+1},g^M)$ that is an immortal type III solution of the Ricci flow, we will note it $(\mathbb{R}^{n+1},g^{M(t)})$. This flow is asymptotic to the cone $C(N)$ at each time.

\item From this Ricci flow, it is possible to get a sublimit of a blowdown sequence. This will give the wanted Ricci flow coming out of the cone $N$. We expect the resulting Ricci flow to be an expanding soliton (it is true in some cases). 
\end{itemize}
We will focus on a family of manifolds whose Perelman's $\mathcal{W}$-functional is very close to the $\mathcal{W}$-functional of the sphere (which is the crucial point of our proof).
\begin{thm}
There exist $\beta_1(n)$ and $\beta_2(n)$ ($0<\beta_1<1<\beta_2$) such that :
\\
\indent If $(\mathbb{S}^n,g^N)$ is a $n$-dimensional Riemannian manifold satisfying the following set of properties $P(\beta_1,\beta_2)$ :
\begin{itemize}
\item $(N,g^N)\times \mathbb{R}^2$ has positive isotropic curvature (implied by the positiveness of the curvature tensor for example).
\begin{rem}
We know by the sphere theorem of Brendle and Schoen that this implies that $N$ is diffeomorphic to $\mathbb{S}^n$, so the initial supposition is redundant.
\end{rem}
\item $C^0$-closeness :
\begin{align*}
\beta_1^2g^{\mathbb{S}^n}\leqslant g^N \leqslant \beta_2^2g^{\mathbb{S}^n}.
\end{align*}
\item Lower bound on the scalar curvature :
\begin{align*}
\R^N\geqslant \frac{n(n-1)}{\beta_2^2}, 
\end{align*}
\end{itemize}
\indent Then, there exists an immortal type III solution of the Ricci flow coming out of the cone $C(N)$.
\end{thm}

\indent These are quite artificial conditions that are just chosen in order to be able smooth out the cone by a manifold of high $\nu$-functional thanks to an adaptation of the proof of the lower bound for the $\mathcal{W}$-functional for perturbations of the sphere.
\\

 The proof of this theorem is the goal of the section.

\subsection{Smoothing out cones while controlling the $\nu$-functional}
Here we present the first step of the proof : smoothing out the cone by a smooth manifold of high $\nu$-functional. This is the section in which we use the quite artificial assumptions on the link.
\\

Let us start the proof by proving that the condition we asked for is preserved along a renormalization of Ricci flow with uniform worse constants.
\begin{defn}
We will call $(\tilde{g}_t)_t$ a \textit{renormalization of the Ricci flow} on $N$ if there exists a time depending constant $\alpha(t)$ such that :
$$\partial_t\tilde{g} = -2\Ric(\tilde{g}) + \alpha \tilde{g}.$$
\end{defn}

\begin{prop}\label{Prop preservée}
For all $N$ satisfying $P(\beta_1,\beta_2)$, there exists a renormalization of the Ricci flow $(\tilde{g}(t))_{t\in[0,+\infty)}$ on $\mathbb{S}^n$ starting at $g^N$ that exists for all times and such that as the time tends to $\infty$, $$\tilde{g_t}\to g^{\mathbb{S}^n}.$$

 Moreover, for all $t>0$, $\tilde{g}_t$ satisfies $P(\beta_1',\beta_2')$ for some other $\beta'_1$ and $\beta'_2$ depending only on the initial $\beta_1$ and $\beta_2$ and the dimension.
\\

 Moreover, as $\beta_1$ and $\beta_2$ tend to $1$, $\beta'_1$ and $\beta'_2$ also tend to $1$.
\end{prop}
\begin{proof}
This result relies completely on the proof of the article \cite{BM} by Bamler and Maximo and on the $C^0$ perturbation of Ricci flows presented in \cite{Sim} :

 Let us first note that, by \cite{bs}, the positive isotropic curvature when crossed with $\mathbb{R}^2$ is preserved by Ricci flow and that the convergence towards $\left(\mathbb{S}^n,T_N\times2(n-1)g^{\mathbb{S}^n}\right)$ by the flow defined in the next lemma is exponentially fast in every derivative.

\paragraph{1) A lower bound on the scalar curvature}
\begin{lem}
There exists $\beta_1'(\beta_1,\beta_2)$ and $\beta_2'(\beta_1,\beta_2)$ which tends to $1$ as $(\beta_1,\beta_2)\to (1,1)$ such that if $(\mathbb{S}^n,g^N)$ satisfies $P(\beta_1,\beta_2)$, 

then, defining $\bar{g}(t)$ by :
$$
\left\{
\begin{array}{ll}
\bar{g}(0)=g^N,\\
\partial_t\bar{g} = -2\Ric(\bar{g})+\frac{1}{2T_N}\bar{g},
\end{array}
\right.
$$
where $T_N$ is the existence time of the flow.
\\

$ \bar{N}(t) = (\mathbb{S}^n,\bar{g}(t))$ satisfies $$R^{\bar{N}(t)}\geqslant \frac{n(n-1)}{{\beta_2'}^2}$$ along the flow.
\end{lem}
\begin{proof}
 Let us start by proving that the scalar curvature does not become to negative along the actual Ricci flow. 
\\

 Because of the $C^0$-closeness, the extinction time is close to that of the unit sphere by the work of \cite{Sim} because of the $C^0$-closeness, so we can use the results in \cite{BM} which deals with manifolds with positive isotropic curvature when crossed with $\mathbb{R}^2$ that have a lower bound on the scalar curvature as well as an existence time close to the maximum possible given the lower bound on the scalar curvature :

 To be more precise, we will need to look at the proof given in \cite{BM}.

 In the section "End of Proof of Theorem 1.1.". after an arbitrarily short time noted $t_2$ (depending only on the lower bound on the scalar curvature and on the existence time), the sectional curvatures are very pinched around a lower bound for the scalar curvature (with a factor $\frac{1}{n(n-1)}$). Thanks to \cite{hui}, this is a preserved property along the flow. 
\\

 This implies that $\frac{\R^2}{n}\leqslant|\Ric|^2\leqslant\frac{(1+\epsilon)\R^2}{n}$ and since the equation satisfied by the scalar curvature along a Ricci flow is $$\partial_t R = \Delta R + 2 |Ric|^2,$$ we can deduce that the product $\R^N_{min}\times (T_N-t)$ (where $T_N$ is the existence time of the flow) satisfies along a Ricci flow :
$$\frac{n}{2}-\mathcal{O}(\epsilon)\leqslant \R^N_{min}\times (T_N-t)\leqslant \frac{n}{2}$$
(the lower bound comes from a lower bound on the existence time thanks to the maximum of the scalar curvature).

 This bound implies that along the renormalized flow presented in the theorem, we have $\R^N_{min}\geqslant n(n-1)-\mathcal{O}(\epsilon)$ with $\epsilon$ a function of $(\beta_1,\beta_2)$ which tends to $0$ as $(\beta_1,\beta_2)$ tends to $(1,1)$.
\end{proof}

\paragraph{2) A bound on the $C^0$-closeness}

\begin{lem}
	There exists $\beta_1'(\beta_1,\beta_2)$ and $\beta_2'(\beta_1,\beta_2)$ which tends to $1$ as $(\beta_1,\beta_2)\to (1,1)$ such that if $(\mathbb{S}^n,g^N)$ satisfies $P(\beta_1,\beta_2)$, 
	
	then defining $\bar{g}(t)$ by :
	$$
	\left\{
	\begin{array}{ll}
	\bar{g}(0)=g^N,\\
	\partial_t\bar{g} = -2\Ric(\bar{g})+\frac{1}{2T_N}\bar{g},
	\end{array}
	\right.
	$$
	where $T_N$ is the existence time of the flow.
	\\
	
	$ \bar{N}(t) = (\mathbb{S}^n,\bar{g}(t))$ satisfies $${\beta_1'}^2 g^{\mathbb{S}^n}  \leqslant \bar{g}(t)\leqslant {\beta_2'}^2 g^{\mathbb{S}^n} $$ along the flow.
\end{lem}

This is actually a direct application of the theorem 1.1 of \cite{BM} : Since we have a lower bound on $\R_{min}(t)\times (T_N-t)$ we have the $C^0$-closeness preserved by scaling (with a potentially worse constant).

\paragraph{3) Converging to the unit sphere}
Now, in the statement of the proposition, we asked for the flow to converge to the unit sphere and not one of radius $\sqrt{T_N\times 2(n-1)}$. It is possible to do by just scaling $g^N$ at the end of the flow where it is $C^3$-close to a sphere as it is presented in the appendix B.2. It is possible to do this without changing the constant $\beta_1'$ and $\beta_2'$ as $\bar{g}$ is $C^{2}$-close to $T_N\times 2(n-1)$. 

We will name $\tilde{g}_t$ the resulting flow.
\\
\end{proof}

Let us now use the proposition \ref{Prop preservée} to smooth out the cone over $(N,g^N)=(\mathbb{S}^n,g^N)$ thanks to a manifold of the form : 
$$(\mathbb{R}^+\times \mathbb{S}^n,dr^2+r^2\hat{g}(r)),$$
and prove that for a well chosen $\hat{g}(r)$, the $\nu$-functional is large. 

We will actually choose $$\hat{g}(r) = \tilde{g}(\phi(r))$$ for a $\phi:\bar{\mathbb{R}}^+\to\bar{\mathbb{R}}^+$. We want the manifold to be smooth, that is to have $\hat{g}(r)$ asymptotic to the unit sphere for $r$ close to $0$, so $\phi(0) = +\infty$ and we also want the manifold to be asymptotic to $C(N)$, that is $\phi(+\infty) = 0$. To have good estimates coming from the work we have already done on cones, we will consider such a function with slow enough variations.

\begin{lem}
For all $\delta$, the manifold $(M,g^\delta):=\left(\mathbb{R}^+\times \mathbb{S}^n,dr^2 + r^2 \tilde{g}(\frac{\delta}{r^2})\right)$ smoothes out the cone $C(N)$ and for $\delta_0(N)>0$ small enough, we have $\nu^M(g^{\delta_0})>-\eta_n$.

 In other words, it is possible to smooth out the cone by a manifold that satisfies the assumption of the pseudolocality result of the previous section.
\begin{rem}
A small $\delta$ corresponds to small variations of the link to make it look constant locally and thus look locally like a cone from the point of view of the $\mathcal{W}$-functional (it is constant if and only if $(M,g^{\delta})$ is a cone).
\end{rem}
\begin{proof}
Let us consider a manifold $(M,g^\delta)=(\mathbb{R}^+\times \mathbb{S}^n,dr^2+r^2\tilde{g}(\frac{\delta}{r^2}))$.

\begin{lem}
For all $t$, $\tilde{g}_t$ satisfies a lower bound $L(\epsilon_1,\epsilon_2,\epsilon_3)$ for some $\epsilon_i$ tending to zero as $\beta_i\to 0$.
\end{lem}
\begin{proof}
Manifolds satisfying $P(\beta_1',\beta_2')$ for $(\beta_1',\beta_2')$ close enough to $(1,1)$ also satisfy $L(\epsilon_1,\epsilon_2,\epsilon_3)$ for small $\epsilon_i$ thanks to the appendix A.4.
\end{proof}
Now, thanks to the appendix A.3, such a manifold $(M,g^\delta)$ can be treated as a cone over a a manifold satisfying a lower bound $L(\epsilon_1+\mathcal{O}(\delta),\epsilon_2,\epsilon_3)$ (in the sense that its $\nu$-functional is bounded from below by the same $\Psi(\epsilon_1+\mathcal{O}(\delta),\epsilon_2,\epsilon_3)$).
\end{proof}
\end{lem}

\subsection{A type III immortal solution of the Ricci flow asymptotic to the cone}
Thanks to the last proposition, for all $N$ satisfying the assumptions of the theorem, then it is possible to construct a manifold asymptotic to the cone $C(N)$ 
\\

\begin{prop}
For $N$ satisfying the set of properties $P(\beta_1,\beta_2)$, the Ricci flow starting at $(M,g^{\delta_0})$ is an immortal type III solution of the Ricci flow. It is moreover asymptotic to the cone $C(N)$ at all times.
\end{prop}
\begin{proof}
The existence of the type III solution starting at $(M,g)$ is a direct consequence of the pseudolocality result presented in the previous section, because for $(\beta_1,\beta_2)$ close enough to $(1,1)$, we have :
$$\nu^M(g)>\eta_{n+1}.$$

\indent The fact that the flow stays asymptotic to the cone can be found in \cite{lz}.
\end{proof}
\subsection{Construction of a type III immortal solution coming out of the cone by a blow down process}
From the previous results, we have a type III solution of the Ricci flow. A standard process to get an idea of the long time behavior of such a Ricci flow is to "blow down" the flow by parabolic scaling. 
\\

 Thanks to the results in the appendix D, we can take a sublimit of the blow downs of this Ricci flow into an immortal type III solution of the Ricci flow coming out of a cone. Which is the statement of the theorem.

\begin{rem}
Note that if we had a limit instead of a sublimit, we could argue that the limiting flow is an expanding soliton. We could expect that such a Ricci flow would be a soliton.

 For example, if we can ensure the positivity of the Ricci curvature along the renormalized flow, then a direct application of the theorem 1 of \cite{Ma} ensures that we have a gradient expanding soliton. It could be possible to use the same techniques in our case, this would require to get controls of the heat kernels on manifolds satisfying a log sobolev inequality (and a lower bound on the scalar curvature). In particular, assuming that the link $N$ also has a curvature tensor larger than that of the unit sphere, we obtain an expanding soliton.
\end{rem}

\newpage
\appendix
\section{Proof of some lemmas}
Let us give the proof of some technical and not particularly enlightening lemmas needed for some proofs in the text
\subsection{Upper semicontinuity of $\tau\mapsto\mu(g,\tau)$}
\begin{lem}
$\tau\mapsto\mu(g_0,\tau)$ when $\lambda^N>-\infty$ is upper semicontinuous.
\\
\indent More precisely,
\begin{align*}
\lim_{\tau_2\to\tau_1}\left(\frac{\mu(g_0,\tau_1)-\mu(g_0,\tau_2)}{\tau_1-\tau_2}\right)&\geqslant \mathcal{F}(u_1)-\frac{n}{2\tau_1}\geqslant \lambda^N-\frac{n}{2\tau_1}.
\end{align*}
\begin{proof}
Let us consider $\tau_1>\tau_2>0$.
\\
\\
Let $\phi : N\to \mathbb{R}$ such that $\int_N \phi^2dv = 1$.
\\
For all $u$ function on the manifold, we have :
\begin{align*}
\mathcal{W}\left(\phi-\frac{n}{2}\log(4\pi\tau_1),g_0,\tau_1\right)-\mathcal{W}\left(\phi-\frac{n}{2}\log(4\pi\tau_2),g_0,\tau_2\right)&=(\tau_1-\tau_2)\mathcal{F}(\phi,g_0)-\frac{n}{2}\log\frac{\tau_1}{\tau_2}\\
&\geqslant (\tau_1-\tau_2)\lambda^N-\frac{n}{2}\log\frac{\tau_1}{\tau_2}.
\end{align*}
Now, if we consider $\phi_1$ for which $\mathcal{W}(\phi_1-\frac{n}{2}\log(4\pi\tau_1), g_0, \tau_1)=\mu(g_0,\tau_1)$ (or approximating it up to a $o(\tau_1-\tau_2)$ in the case where such a minimizer doesn't exist).
\begin{align*}
\mu(g_0,\tau_1)-\mathcal{W}\left(\phi_1-\frac{n}{2}\log(4\pi\tau_2),g_0,\tau_2\right)&\geqslant (\tau_1-\tau_2)\lambda^N-\frac{n}{2}\log\frac{\tau_1}{\tau_2}\\
\mu(g_0,\tau_1)&\geqslant \mathcal{W}\left(\phi_1-\frac{n}{2}\log(4\pi\tau_2),g_0,\tau_2\right)+ (\tau_1-\tau_2)\lambda^N-\frac{n}{2}\log\frac{\tau_1}{\tau_2}\\
\mu(g_0,\tau_1)&\geqslant \mu(g_0,\tau_2)+ (\tau_1-\tau_2)\lambda^N-\frac{n}{2}\log\frac{\tau_1}{\tau_2}.
\end{align*}

Now, since $(\tau_1-\tau_2)\lambda^N-\frac{n}{2}\log\frac{\tau_1}{\tau_2}$ tends to $0$ when $\tau_2$ tends to $\tau_1$, we have the wanted upper semicontinuity.
\end{proof}
\end{lem}
\subsection{A sequence of minimizers of $\mathcal{W}(.,g,\tau)$ tends to a minimizer of $\mathcal{F}$ as $\tau\to\infty$}
\begin{prop}
For any compact manifold $N$, as $\tau_k$ tends to infinity, a sequence $u_k^2=\frac{e^{-f_k}}{(4\pi\tau_k)}$, where $f_k$ is a minimizer of $\mathcal{W}^N(.,g,\tau_k)$ tends to $u_\mathcal{F}$ in $H^1$, the minimizer of $\mathcal{F}^N(.,g)$. 
\begin{proof}
Let $\tau>0$, $f: N \to \mathbb{R}$.
\\
Noting $u^2 = e^{-\phi}$, where $\phi = f+\frac{n}{2}\log(4\pi\tau)$.
\\
\indent The expression of the entropy on $N$ is :
\begin{align}
\mathcal{W}^N\left(f,g,\tau\right) = \tau\mathcal{F}^N(\phi,g)-\int_Nu^2\log(u^2)dv-\frac{n}{2}\log{4\pi\tau}-n.\label{entrexp}
\end{align}
It is expected that a minimizer of $\mathcal{W}^N(.,g,\tau)$ becomes close to a minimizer of $\mathcal{F}^N(.,g)$ as $\tau$ tends to $\infty$ since the importance of the compensating term $-\int_Nu^2\log(u^2)dv$ becomes neglectible compared to the $\tau\mathcal{F}^N(\phi,g)$.
\paragraph{1) Let us bound the $u^2\log u^2$ term}
:
\\
There exists a $\tau_0$ such that, for all $u$ such that $\int_N u^2 = 1$ (this exists because the manifold is compact),
\begin{align*}
-\int_N u^2\log (u^2)dv &\geqslant -\tau_0\int_N4|\nabla u|^2dv.
\end{align*}
And, by Jensen inequality :
\begin{align}
0\geqslant-\int_N u^2\log (u^2)dr. \label{up bound log}
\end{align}
\\
\textbf{Let us consider a sequence $(u_k)_k$ of minimizers for $\tau_k\to+\infty$}
\\
From these two inequalities, can bound the $u^2\log (u^2)$ term and we get :
$$\int_N \left(4|\nabla u|^2+\R u^2\right)dv\geqslant\frac{\mathcal{W}^N(u,\tau_k)+\frac{n}{2}\log(4\pi\tau_k)+n}{\tau_k} \geqslant \int_N \left(\frac{(\tau_k-\tau_0)}{\tau_k}4|\nabla u|^2+\R u^2\right)dv,$$
and a minimizer of the middle term is also a minimizer of $\mathcal{W}$.
\paragraph{2) Let us prove that $u_k$ tends to a minimizer of $\mathcal{F}$ by comparing its $\mathcal{W}$-functional to that of $u_\mathcal{F}$, a minimizer of the $\mathcal{F}$ functional}
:
$$\mathcal{F}(u_\mathcal{F}) = \lambda^N.$$
\\
Since $u_k$ is a minimizer of $\mathcal{W}^N(.,\tau_k)$, we have :

\begin{align}
0 &\leqslant \mathcal{W}^N(u_\mathcal{F},\tau_k)-\mathcal{W}^N(u_k,\tau_k) \nonumber\\
&\leqslant \tau_k (\lambda^N-\mathcal{F}(u_k)) -\frac{\tau_0}{\tau_k}\int_N |\nabla u_k|^2dv. \label{W F}
\end{align}
Now, if $u_k$ doesn't approach a minimizing function for $\mathcal{F}^N$, then there exists $\epsilon>0$ such that, for all $k$ :
$$\lambda^N-\mathcal{F}(u_k)<-\epsilon,$$
and, as $\tau\to\infty$, the right term of \eqref{W F} tends to $-\infty$ (the other terms are negative or tend to $0$ as $k\to\infty$), which is a contradiction of the inequality.
\\

$u_k$ tends to a minimizer of $\mathcal{F}$ in $H^1$.
\end{proof}
\end{prop}
\subsection{Making the manifold look locally conical}
Here we want to prove that given a flow on a manifold diffeomorphic to a sphere that ends at the unit sphere, it is possible to construct a manifold asymptotic to the cone over the initial manifold for which at each point, the quantities involved in the entropy are arbitrarily close to that over a cone. 
\\
\indent This makes the computation of the entropy much easier, in particular to prove that the $\nu$-functional of the manifold smoothing the cone is high enough in the last section of the text.
\\
\indent We will consider a manifold : 
$$(M,g) = (\mathbb{R}^+\times N,dr^2+r^2g(r)).$$
In the case when $g(r)$ is constant, we have a cone and the expression from which we have been able to control the $\nu$-functional after a separation of variables is :
\begin{align*}
\mathcal{W}^{C(N)}(f,g,\tau) &= \int_0^\infty \left[\mathcal{W}^{N}\left(\tilde{f},g^N,\frac{\tau}{r^2}\right)-n(n-1)\frac{\tau}{r^2}\right]\left(\frac{e^{-a_r}}{(4\pi\tau)^{\frac{1}{2}}} dr\right) \\
&+\int_N \left(\int_0^\infty \left[\tau(\partial_r f)^2+a_r-1\right]\left(\frac{e^{-a_r}}{(4\pi\tau)^{\frac{1}{2}}} dr\right)\right)\left(\frac{e^{-\tilde{f}}}{\left(4\pi\tau r^{-2}\right)^{\frac{n}{2}}} dv\right).
\end{align*}
The goal is to have a similar expression (up to a controlled error term) for manifolds of the form $(M,g) = (\mathbb{R}^+\times N,dr^2+r^2g(r))$ for an adapted choice of $g(r)$, that is :
\begin{align*}
\mathcal{W}^{M}(f,g,\tau) &= \int_0^\infty \left[\mathcal{W}^{N}\left(\tilde{f},g(r),\frac{\tau}{r^2}\right)-n(n-1)\frac{\tau}{r^2}\right]\left(\frac{e^{-a_r}}{(4\pi\tau)^{\frac{1}{2}}} dr\right) \\
&+\int_N \left(\int_0^\infty \left[\tau(\partial_r f)^2+a_r-1\right]\left(\frac{e^{-a_r}}{(4\pi\tau)^{\frac{1}{2}}} dr\right)\right)\left(\frac{e^{-\tilde{f}}}{\left(4\pi\tau r^{-2}\right)^{\frac{n}{2}}} dv\right) \\
&+ \text{Error term}.
\end{align*}
\\

 We will choose $g(r) = g^{\hat{N}(\theta(r))}$ for a $\theta$ with slow enough variations to make $\R^M$ arbitrarily close to $\R^{C(\hat{N}(\theta(r)))}$ at a point at distance $r$. 

 The idea is to consider the variations of $g(r)$ small enough to consider it locally conical in the expression of the $\mathcal{W}$-functional. By taking advantage of the exponentially fast convergence of the metric along a renormalized flow.
\paragraph{Some formulas for Riemannian foliations by hypersurfaces}
:
\\
Now, we have the following formula for the curvature of $M =(\mathbb{R}\times N, dt^2 + g_t)$ with 
$$\partial_t g_t = 2K(t) :$$

For the Ricci curvature :
\begin{itemize}
\item $\Ric_{00} = -\partial_t K^i_i + g^{ij}K_{ij}$,
\item $\Ric_{0i} = -D_i K^j_j + D_jK^j_i$,
\item $\Ric_{ij} = \Ric^{n(t)}_{ij}-\partial_t K_{ij}+2 K_{il}K^l_j-K^l_lK_{ij}$.
\end{itemize}
Which gives that the scalar curvature is :
\begin{align}
\R^M = \R^{N(t)}-\partial_t K^i_i + g^{ij}K_{ij} + \sum_{i,j}g^{ij}(-\partial_t K_{ij}+2 K_{il}K^l_j-K^l_lK_{ij}).
\label{exp scal}
\end{align}
\\
\paragraph{Estimating the difference with the expression of $\mathcal{W}$ for a cone}
:
\\
\indent Let us define : 
$$\R_{rest} := \R^M(r)-\R^{C(\hat{N}(\theta(r)))}(r).$$
\begin{lem}
If the convergence of the family of metric considered is exponentially fast in the $C^2$-sense.

Choosing $\theta(r) = \frac{\delta}{r^2}$, we have the following estimate for small $\delta$ : 
\begin{align*}
\R_{rest}(r)=\mathcal{O}(\delta)
\end{align*}
uniformly on the manifold.
\end{lem}
\begin{rem}
	The convergence will always be exponentially fast for the constructions by renormalized Ricci flow we will consider.
\end{rem}
\begin{proof}

 Thanks to \eqref{exp scal}, we have an explicit expression for the scalar curvature in our case depending on the first and second derivatives
\\

 Let us just note that this quantity is a $O\left(|\partial^2_{r^2}g(r)|+\frac{1}{r}|\partial_r g(r)|+\frac{1}{r^2}|\partial_r g(r)|^2\right)$.
\\

Now, we know that the convergence of $\hat{g}$ is exponentially fast in the $C^2$-sense :
$$|\partial^k_{t^k}\hat{g}|\leqslant C_ke^{-c_kt},$$
so, choosing $\theta(r)=\frac{\delta}{r^2}$ ($g(r) = \hat{g}(\frac{\delta}{r^2})$, for some small $\delta$), by the $O\left(|\partial^2_{r^2}g(r)|+\frac{1}{r}|\partial_r g(r)|+\frac{1}{r^2}|\partial_r g(r)|^2\right)$ estimate, we have :
\begin{align*}
\R_{rest} &= O\left(|\partial^2_{r^2}g(r)|+\frac{1}{r}|\partial_r g(r)|+\frac{1}{r^2}|\partial_r g(r)|^2\right)\\
&=O\left(\left((\theta')^2|\partial^2_{r^2}\hat{g}|+\theta''|\partial_r\hat{g}|\right)+\left(\frac{1}{r}\theta'|\partial_r \hat{g}|\right)+\left(\frac{1}{r^2}(\theta')^2|\partial_r \hat{g}|^2\right)\right)\\
&=O\left(\left(\frac{\delta^2}{r^6}e^{-\frac{\delta c_2}{r^2}}+\frac{\delta}{r^4}e^{-\frac{\delta c_1}{r^2}}\right)+\left(\frac{\delta}{r^4}e^{-\frac{\delta c_1}{r^2}}\right)+\left(\frac{\delta^2}{r^8}e^{-\frac{2\delta c_1}{r^2}}\right)\right)\\
&=\mathcal{O}(\delta+\delta^2).
\end{align*}
uniformy in $r$.
\end{proof}
So for $\delta$ small enough, this term will become small. Let us see how we can take care of it :
\begin{lem}
\begin{align*}
\mathcal{W}^{M}(f,g,\tau) &= \int_0^\infty [\tilde{\mathcal{W}}^{N(r)}(\tilde{f},g(r),\tau r^{-2})]\left(\frac{e^{-a_r}}{(4\pi\tau)^{\frac{1}{2}}} dr\right)\\
&+\int_N \int_0^\infty [\tau(\partial_r f)^2+a_r-1]\left(\frac{e^{-a_r}}{(4\pi\tau)^{\frac{1}{2}}} dr\right)\left(\frac{e^{-\tilde{f}}}{\left(4\pi\tau r^{-2}\right)^{\frac{n}{2}}} dv\right) \\
&+ \mathcal{O}(\delta)\int_0^{\infty}\frac{\tau}{r^2}\left(\frac{e^{-a_r}}{(4\pi\tau)^{\frac{1}{2}}} dr\right).\\
\end{align*}
\begin{proof}
Since on $M$ we have the following formulas coming from the foliations formulas :
\begin{itemize}
\item $\R^M = \frac{\R_{N(r)}-n(n-1)}{r^2}+\R_{rest}$ by definition of $\R_{rest}$,
\item $|\nabla^M f|^2 = (\partial_r f)^2+\frac{|\nabla^{N(r)} f|^2}{r^2}$,
\item $dv^M = r^ndv^{N(r)}dr$.
\end{itemize}
This implies that $\mathcal{W}^M$ has the following expression :
\begin{align*}
\mathcal{W}^M(f,g,\tau) =\int_0^\infty\int_{N(r)} &\left[\tau\left((\partial_r f)^2+\frac{|\nabla^{N(r)} f|^2+(\R^{N(r)}-n(n-1)+\R_{rest})}{r^2}\right)+f-(n+1)\right]\\
&\times\left(\frac{e^{-a_r}}{(4\pi\tau)^{\frac{1}{2}}} dr\right)\left(\frac{e^{-\tilde{f}}}{\left(4\pi\tau r^{-2}\right)^{\frac{n}{2}}} dv\right).
\end{align*}
\indent Now, defining the same separation of variables as in the cone case and using the last lemma stating that that $\R_{rest} = \mathcal{O}(\delta)$ we get :
\begin{align*}
\mathcal{W}^{M}(f,g,\tau) &= \int_0^\infty [\tilde{\mathcal{W}}^{N(r)}(\tilde{f},g(r),\tau r^{-2})]\left(\frac{e^{-a_r}}{(4\pi\tau)^{\frac{1}{2}}}dr\right) \\
&+\int_N \int_0^\infty [\tau(\partial_r f)^2+a_r-1]\left( \frac{e^{-\tilde{f}}}{(4\pi\tau r^{-2})^{\frac{n}{2}}}dv\right)\left(\frac{e^{-a_r}}{(4\pi\tau)^{-\frac{1}{2}}}dr\right)  \\
&+ \mathcal{O}(\delta)\int_0^{\infty}\frac{\tau}{r^2}\left(\frac{e^{-a_r}}{(4\pi\tau)^{\frac{1}{2}}}dr\right).\\
\end{align*}
\end{proof}
\end{lem}

\subsection{Lower bounds on the $\mathcal{W}$-functional of perturbation of the sphere}
In this section, we see under which conditions, we can control the $\mathcal{W}$-functional of perturbation of the sphere.
\\
\indent We will obtain a precise enough control for our purpose by assuming a lower bound on the scalar curvature and the $C^0$-closeness. Recall that the $\mathcal{W}$-functional has three components :
$$\mathcal{W}(f,g^N,\tau) = \tau \mathcal{F}\left(f+\frac{n}{2}\log(4\pi\tau)\;,\;g\right)+\mathcal{N}\left(f+\frac{n}{2}\log(4\pi\tau)\;,\;g\right) +C(n,\tau).$$
We will see how we can control each of these.
\begin{prop}
If $(\mathbb{S}^n, g^N)$ satisfies the two following properties :
\begin{itemize}
\item $C^0$-closeness : 
$$\beta_1^2 g^{\mathbb{S}^n}\leqslant g^N \leqslant \beta_2^2g^{\mathbb{S}^n}.$$
\item Lower bound on the scalar curvature : 
$$\R^N\geqslant \frac{n(n-1)}{\beta_2^2}.$$
\end{itemize}
Then, we have the following lower bounds on the different component of the $\mathcal{W}$-functional compared to that of $g^{\mathbb{S}^n}$ :
\begin{itemize}

\item For $\mathcal{F}$ :
$$\mathcal{F}\left(f+\frac{n}{2}\log(4\pi\tau)\;,\;g^N\right)\geqslant \frac{\beta_1^{n}}{\beta_2^{n+4}}\mathcal{F}\left(f +\delta+\frac{n}{2}\log(4\pi\tau)\;,\;g^{\mathbb{S}^n}\right).$$
\item For $\mathcal{N}$ :
$$\mathcal{N}\left(f+\frac{n}{2}\log(4\pi\tau)\;,\;g^N\right)\geqslant \frac{\beta_2^{n}}{\beta_1^n}\mathcal{N}\left(f+\delta+\frac{n}{2}\log(4\pi\tau)\;,\;g^{\mathbb{S}^n}\right)+\left(\frac{\beta_1^{n}}{\beta_2^n}-\frac{\beta_2^{n}}{\beta_1^n}\right)\frac{vol(\mathbb{S}^n)}{e}-n\log\beta_2.$$
\end{itemize}
Where $\delta$ is defined to ensure that $\int_{\mathbb{S}^n}\frac{e^{-f-\delta}}{(4\pi\tau)^{\frac{n}{2}}}dv^{\mathbb{S}^n} = 1$. Note that it depends on $f$, but that there are bounds on it that only depend on the $C^0$-closeness.
\end{prop}
\begin{proof}
All along the proof, we will note $\theta:=\frac{dv^N}{dv^{\mathbb{S}^n}}$. By the $C^{0}$-closeness, we have $\beta_1^n\leqslant\theta\leqslant\beta_2^n$. Let us also note that the $\delta$ defined in the statement of the proposition satisfies 
$$-n\log\beta_2\leqslant \delta\leqslant -n\log\beta_1$$ 

 Let us start by proving the estimate for $\mathcal{F}$ : 
\\

 We have a lower bound on the scalar curvature and the volume form, we only have to control $g^N\left(\nabla^N f,\nabla^N f\right)$ thanks to $g^{\mathbb{S}^n}\left(\nabla^{\mathbb{S}^n} f,\nabla^{\mathbb{S}^n} f\right)$ :
\\

 By definition, we have, for any $v$ in $T_x\mathbb{S}^n$ for some $x\in\mathbb{S}^n$ : 
\begin{align*}
df(v) &= g^N\left(\nabla^N f,v\right)\\
&= g^{\mathbb{S}^n}\left(\nabla^{\mathbb{S}^n} f,v\right).
\end{align*}
We can decompose $\nabla^N f =: \alpha \nabla^{\mathbb{S}^n} f +p^N(\nabla^{N}f)$ where $p^N$ is the projection on the orthogonal of $\nabla^{\mathbb{S}^n} f$ for $g^N$ (note that $\alpha$ depends on the point at which we look at the tangent space and the direction of $\nabla^{\mathbb{S}^n} f$).
\\
\indent Let us find bounds on $\alpha$ which will give us a lower bound on $g^N\left(\nabla^N f,\nabla^N f\right)$ :
By definition, we have :
\begin{align*}
df(\nabla^{\mathbb{S}^n} f) &= g^N\left(\nabla^N f,\nabla^{\mathbb{S}^n} f\right)\\
&= \alpha g^N\left(\nabla^{\mathbb{S}^n} f,\nabla^{\mathbb{S}^n} f\right)\\
&=g^{\mathbb{S}^n}\left(\nabla^{\mathbb{S}^n} f,\nabla^{\mathbb{S}^n} f\right)>0.
\end{align*}
Now, since $\beta_1^2g^{\mathbb{S}^n}\leqslant g^N\leqslant \beta_2^2g^{\mathbb{S}^n}$, we have :
$$\frac{1}{\beta_2^2}\leqslant\alpha\leqslant \frac{1}{\beta_1^2}.$$
\indent We are now ready to bound the $\mathcal{F}$ functional :
\begin{align*}
\mathcal{F}&\left(f+\frac{n}{2}\log(4\pi\tau)\;,\;g^N\right) = \int_{\mathbb{S}^n} \left(g^N\left(\nabla^N f,\nabla^N f\right)+\R^N\right)\frac{e^{-f}}{(4\pi\tau)^{\frac{n}{2}}}\;\theta dv^{\mathbb{S}^n} \\
&\geqslant \int_{\mathbb{S}^n} \left(\alpha^2g^N\left(\nabla^{\mathbb{S}^n} f,\nabla^{\mathbb{S}^n} f\right)+g^N(p^N(\nabla^N f)\;,\;p^N(\nabla^N f))+\frac{n(n-1)}{\beta_2^2}\right)\frac{e^{-f-\delta}}{(4\pi\tau)^{\frac{n}{2}}}\; e^{\delta}\beta_1^n dv^{\mathbb{S}^n}.
\end{align*}
By the decomposition of $\nabla^N f$ and the $C^0$-closeness.
\\
\indent We can now use our bound on $\alpha$ and $\delta$ and conclude that :
\begin{align*}
\mathcal{F}&\left(f+\frac{n}{2}\log(4\pi\tau)\;,\;g^N\right)\geqslant \int_{\mathbb{S}^n} \left(\frac{1}{\beta_2^4}g^N\left(\nabla^{\mathbb{S}^n} f+\delta,\nabla^{\mathbb{S}^n} f+\delta\right)+\frac{n(n-1)}{\beta_2^2}\right)\frac{e^{-f-\delta}}{(4\pi\tau)^{\frac{n}{2}}}\;\frac{\beta_1^{n}}{\beta_2^n}dv^{\mathbb{S}^n}\\
&\geqslant \frac{\beta_1^{n}}{\beta_2^{n+4}}\mathcal{F}\left(f+\delta+\frac{n}{2}\log(4\pi\tau)\;,\;g^{\mathbb{S}^n}\right).
\end{align*}
Which is the stated inequality.
\\
\\
\indent Let us now take care of the $\mathcal{N}$ functional.
\\
\indent Let us note $u = \frac{e^{-f}}{(4\pi\tau)^{\frac{n}{2}}}$ and $w = \frac{e^{-f-\delta}}{(4\pi\tau)^{\frac{n}{2}}}$, which imply :
$$\int_{\mathbb{S}^n}u\; \theta dv^{\mathbb{S}^n} =1,$$
and,
$$\int_{\mathbb{S}^n}w \;dv^{\mathbb{S}^n} =1.$$
The expressions to compare are :
$$\mathcal{N}\left(f+\frac{n}{2}\log(4\pi\tau)\;,\;g^N\right) = -\int_{\mathbb{S}^n}u\log u\; \theta dv^{\mathbb{S}^n},$$
and,
$$\mathcal{N}\left(f+\delta+\frac{n}{2}\log(4\pi\tau)\;,\;g^{\mathbb{S}^n}\right) = -\int_{\mathbb{S}^n}w\log w \;dv^{\mathbb{S}^n}.$$
And note that : $u= e^\delta w$.
\\
\indent Let us start the comparison of both expression by separating the positive and negative part of the integrand :
\begin{align*}
\mathcal{N}&\left(f+\frac{n}{2}\log(4\pi\tau)\;,\;g^N\right) = -\int_{\mathbb{S}^n}u\log u\; \theta dv^{\mathbb{S}^n}\\
&=-e^\delta\int_{\mathbb{S}^n}w\log w\; \theta dv^{\mathbb{S}^n} + \delta\\
&=-e^\delta\int_{\{w\geqslant 1\}}w\log w\; \theta dv^{\mathbb{S}^n}-e^\delta\int_{\{w < 1\}}w\log w \; \theta dv^{\mathbb{S}^n} + \delta.\\
\end{align*}
Now, we have bounds on $\delta$ and $\theta$ coming from the $C^0$-closeness to the unit sphere, namely :
$$-n\log \beta_2\leqslant\delta\leqslant -n\log \beta_1,$$
and,
$$\beta_1^n\leqslant\theta\leqslant\beta_2^n.$$
\indent Now that we have separated the positive and negative part, we can use our upper or lower bounds on each term :
\begin{align*}
\mathcal{N}&\left(f+\frac{n}{2}\log(4\pi\tau)\;,\;g^N\right)= -e^\delta\int_{\{w\geqslant 1\}}w\log w\; \theta dv^{\mathbb{S}^n}-e^\delta\int_{\{w < 1\}}w\log w\; \theta dv^{\mathbb{S}^n} + \delta\\
&\geqslant -\beta_2^{2n}\int_{\{w\geqslant 1\}}w\log w \; dv^{\mathbb{S}^n}-\beta_1^{2n}\int_{\{w < 1\}}w\log w\; dv^{\mathbb{S}^n}+n\log\beta_1.
\end{align*}
We can now use the fact that for any positive real number $x$, $x\log x\geqslant -\frac{1}{e}$ :
\begin{align*}
\mathcal{N}&\left(f+\frac{n}{2}\log(4\pi\tau)\;,\;g^N\right)\geqslant -\frac{\beta_2^{n}}{\beta_1^n}\int_{\{w\geqslant 1\}}w\log w \; dv^{\mathbb{S}^n}-\frac{\beta_1^{n}}{\beta_2^n}\int_{\{w < 1\}}w\log w \; dv^{\mathbb{S}^n}+n\log\beta_1\\
&= -\frac{\beta_2^{n}}{\beta_1^n}\int_{\mathbb{S}^n}w\log w \;dv^{\mathbb{S}^n}+\left(\frac{\beta_1^{n}}{\beta_2^n}-\frac{\beta_2^{n}}{\beta_1^n}\right)\int_{\{w < 1\}}w\log w\; dv^{\mathbb{S}^n}+n\log\beta_1 \\
&\geqslant -\frac{\beta_2^{n}}{\beta_1^n}\int_{\mathbb{S}^n}w\log w\; dv^{\mathbb{S}^n}+\left(\frac{\beta_1^{n}}{\beta_2^n}-\frac{\beta_2^{n}}{\beta_1^n}\right)\frac{vol(\mathbb{S}^n)}{e}+n\log\beta_1,
\end{align*}
which is exactly what we stated.
\end{proof}
In particular, such manifolds satisfy the lower bound $L(\epsilon_1,\; \epsilon_2,\;\epsilon_3)$ defined at the end of section 3 with :
$$
\left\{
\begin{array}{lll}
\epsilon_1 = 1-\frac{\beta_1^{n}}{\beta_2^{n+4}},\\
\epsilon_2 = \frac{\beta_2^{n}}{\beta_1^n}-1,\\
\epsilon_3 = \left(\frac{\beta_1^{n}}{\beta_2^n}-\frac{\beta_2^{n}}{\beta_1^n}\right)\frac{vol(\mathbb{S}^n)}{e}+n\log\beta_1.
\end{array}
\right.
$$

\begin{rem}
In the text, to preserve these estimates along a renormalized Ricci flow, we ask for some more stability of the flow and make use of the Li-Yau-Hamilton inequality generalized by Brendle, thus we ask for the positivity of the isotropic curvature when crossed with $\mathbb{R}^2$, which is a condition preserved by Ricci flow and implied by a positive curvature tensor. 
\\
\indent It is likely that a lower bound on $\lambda^N$ or $\nu^N$ rather than the scalar curvature is enough to get similar estimates.
\end{rem}
\begin{rem}
For a sphere of radius $\beta$, we have the result with :
\begin{itemize}
\item If $\beta\geqslant 1$, then :
$$\mathcal{W}^{\beta\mathbb{S}^n}(f,\beta^2g^{\mathbb{S}^n},\tau)\geqslant \frac{\tau}{\beta^2}\mathcal{F}^{\mathbb{S}^n}\left(f+C,g^{\mathbb{S}^n}\right)+\mathcal{N}^{\mathbb{S}^n}\left(f+C,g^{\mathbb{S}^n}\right)+\frac{n}{2}\log{4\pi\tau}-n,$$
that is : $1-\epsilon_1 = \frac{1}{\beta^2}$, $\epsilon_2 = 0$ and $\epsilon_3 = 0$.
\item If $\beta\leqslant 1$, then :
$$\mathcal{W}^{\beta\mathbb{S}^n}\left(f,\beta^2g^{\mathbb{S}^n},\tau\right)\geqslant \tau\mathcal{F}^{\mathbb{S}^n}(f+C,g^{\mathbb{S}^n})+\mathcal{N}^{\mathbb{S}^n}(f+C,g^{\mathbb{S}^n})+\frac{n}{2}\log{4\pi\tau}-n+n\log\beta,$$
that is $\epsilon_1 = 0$, $\epsilon_2 = 0$ and $\epsilon_3 = -n\log\beta$.
\end{itemize}
And in dimension $3$, thanks to the explicit $\eta_3$ given in the last section, we can smooth out the cones over a sphere of radius $\beta$ if $\beta\in[0.77,1.05]$.
\end{rem}

\section{Particular case of positively curved Einstein links}
Here we give the proofs of the remarks given along the paper about positively curved Einstein manifolds, with a particular interest in the sphere.
\subsection{The behavior of the $\mu$-functional of positively curved Einstein manifolds}
\paragraph{The minimizing function of $\mathcal{W}^N(.,g,\tau)$ is constant if $\tau\geqslant T_N$}
\begin{lem}
Let us consider $(N,g)$ an Einstein manifold such that, $$\Ric = \frac{1}{2T_N}g,$$ (\emph{of shrinking time $T_N$}), and $\phi_N$ its constant potential function such that : $$e^{-\phi_N} vol(N) = 1,$$
\\
\indent and let us note $\phi^\tau$ a function such that $\phi^\tau -\frac{n}{2}\log(4\pi\tau)$ is a minimizer for $\mathcal{W}^N$ at $\tau$.
\begin{align*}
\mathcal{W}^N\left(\phi^\tau -\frac{n}{2}\log(4\pi\tau),g^N,\tau\right) = \mu^N(g^N,\tau).
\end{align*}
\\
Then, for all $\tau \geqslant T_N$, 
\begin{align*}
\phi^\tau = \phi^{T_N}.
\end{align*}
\begin{rem}
In particular, it is also a minimizer of the $\mathcal{F}^N$-functional (which is consistent with the fact that when $\tau\to\infty$, $f^\tau$ gets closer and closer to a minimizer of the $\mathcal{F}$ functional).
\end{rem}
\begin{rem}
Note also that this is not true for $\tau<T_N$. For example, $\phi^{T_{\mathbb{S}^n}}$ is a constant function, but for small $\tau$, $\phi^{\tau}$ looks like a Gaussian on the sphere (see \cite{CHI}). Also note that the value of $\mathcal{W}^{\mathbb{S}^n}$ with a constant function tends to $+\infty$ when $\tau\to 0$.
\\
\indent It is also false for shrinking solitons.
\end{rem}
\begin{proof}
Let us choose $\tau_0 \geqslant T_N$.
\\
\\
There exists $0\leqslant t_0<T_N$ such that $$\tau_0 = \frac{T_N}{1-\frac{t_0}{T_N}}.$$
\\
Let us also consider $N$ a positively curved Einstein manifold of shrinking time $T_N$, that is :
\begin{align*}
g^N(t) =  \left(1-\frac{t}{T_N}\right)g^N(0).
\end{align*}
\paragraph{Let us compute $\mathcal{W}^N\left(g^N(0),\left[f^{T_N}+\frac{n}{2}\log\frac{\tau_0}{T_N}\right],\tau_0\right)$}
:
\\
The Einstein manifold has a shrinking soliton structure : $\Ric + \mathcal{L}_{\nabla \phi^{T_N}}g = \frac{1}{2T_N}g$), with $\phi^{T_N} = \phi_N$. 
\\
\indent Thus, $\phi^{T_N}$ is also a minimizer of $\mathcal{F}^N$,so we have the following equality :
\begin{align*}
\mu^N(g^N,\tau_0) \leqslant \mathcal{W}^N\left(g^N,\left[\phi^{T_N}-\frac{n}{2}\log(4\pi \tau_0)\right],\tau_0\right) =& \mathcal{W}^N\left(\left[\phi^{T_N}-\frac{n}{2}\log(4\pi T_N)\right],g^N,T_N\right)\\
&+ \lambda^N(\tau_0-T_N)-\frac{n}{2}\log\frac{\tau_0}{T_N}\\
=&\mu(g^N,T_N)+ \lambda^N(\tau_0-T_N)-\frac{n}{2}\log\frac{\tau_0}{T_N}.
\end{align*}
So there is equality in the inequality 
\begin{align*}
\mu^N(\tau) \geqslant \mu^N(T_N)+ \lambda^N(\tau-T_N)-\frac{n}{2}\log\frac{\tau}{T_N},
\end{align*}
so for all $T_N<\tau<\tau_0$, $\phi^\tau$ minimizes $\mathcal{W}^N(.,g^N,\tau)$,
\\
that is :
$$\phi^{\tau_0} = \phi^{T_N}.$$
\end{proof}
\end{lem}

As a consequence :

\begin{cor}
	If $N$ is a positively curved Einstein manifold, such that $$\Ric = \frac{1}{2T_N}g,$$ (\emph{of shrinking time $T_N$}).
	\\
	
	Then, for all $\tau\geqslant T_N$, the minimizing function for $\mathcal{W}^N(.,g^N,\tau)$ is constant, and :
	\begin{align}
	\mu^N(\tau) = \mu^N(T_N)+ \lambda^N(\tau-T_N)-\frac{n}{2}\log\frac{\tau}{T_N}
	\end{align}
	that is :
	
	If $\tau\geqslant T_N$ ,
	
	then :
	\begin{align}
	\mu^N(\tau,g^N) = \tau\lambda^N+\log(vol(N))-\frac{n}{2}\log(4\pi\tau)-n.
	\end{align}
	\indent If $\tau\leqslant T_N$,

 then :
	\begin{align}
	0\geqslant\mu^N(\tau,g^N) \geqslant \mu^N(T_N,g^N) = \frac{n}{2}+\log(vol(N))-\frac{n}{2}\log(4\pi T_N)-n.
	\end{align}
	\begin{note}
		In particular, the $\mu$-functional of a manifold with $\Ric = g$ for large $\tau$ only depends on the volume of the manifold.
	\end{note}
\end{cor}

\subsection{Smoothing rotationally symmetric cones by manifolds of high entropy}
In this appendix, we provide an explicit computation of the smoothing process of the last section of the paper applied to spheres. That will illustrate last step of the construction of a renormalization of the Ricci flow preserving a $P(\beta_1',\beta_2')$ property. Where a ball centered at the "tip" of the manifold is arbitrarily close to that of a cone over a sphere.

Since these are Einstein manifolds, there is no need to use a renormalized flow (they would all be constant), but only the last step of scaling the link.

 The result is :
\begin{lem}
For all $n$, there exists $\beta_1(n)$ and $\beta_2(n)$ such that :
\\
For all $\beta\in(\beta_1,\beta_2)$, the cone $C(\beta\mathbb{S}^n)$ satisfies 
$$\nu^{C(\beta\mathbb{S}^n)}> \eta_n,$$
where $\eta_n$ is the number defined in the global pseudolocality lemma.

 Moreover, there exists $M_\beta$ a manifold smoothing it out such that :
$$\nu^{M_\beta}>\eta_n.$$
\begin{rem}
In dimension $3$, given the value of $\eta_3$, it is possible to choose $\beta_1(3) = \frac{2}{e} \approx 0.74$ and $\beta_2 = \sqrt{\frac{2e}{e+2}} \approx 1.07$.
\end{rem}
\end{lem}
\begin{proof}
The first statement is a direct consequence of the lower bounds found on the $\nu$-functional of cones thanks to some closeness to the unit sphere. So we will now focus on the construction of the manifold $M_\beta$.
\\
\\
\indent We will smoothen the cone by the euclidean space around its tip by warped product : $(M,g) = (\mathbb{R}^+\times \mathbb{S}^n,dr^2+h^2g^{\mathbb{S}^n})$
\\
\indent Let us define $h : [0,\infty) \to \mathbb{R}^+$
\begin{enumerate}
\item $h(r) = \frac{r}{\beta}$ \indent for $0\leqslant r\leqslant 1$,
\item $h(r) = \frac{r}{\beta}+ \frac{(r-1)^2}{2\beta A}$ \indent for $1\leqslant r\leqslant b = 1+A(\beta-1)$,
\item $h(r) = r-b+h(b)$ \indent for $r\geqslant b$.
\end{enumerate}
This gives :
\begin{enumerate}
\item $h'(r) = \frac{1}{\beta}$ \indent for $0\leqslant r\leqslant 1$,
\item $h'(r) = \frac{1}{\beta}+ \frac{r-1}{\beta A}$ \indent for $1\leqslant r\leqslant b = 1+A(\beta-1)$,
\item $h'(r) = 1$ \indent for $r\geqslant b$,
\end{enumerate}
and
\begin{enumerate}
\item $h''(r) = 0$ \indent for $0\leqslant r\leqslant 1$,
\item $h''(r) = \frac{1}{\beta A}$ \indent for $1\leqslant r\leqslant b = 1+A(\beta-1)$,
\item $h''(r) = 0$ \indent for $r\geqslant b$.
\end{enumerate}
Let us now consider $(M,g) = (\mathbb{R}^+\times \beta \mathbb{S}^n,dr^2+h^2g_{\beta\mathbb{S}^n})$ and $f: M\to \mathbb{R}$ such that $\int_M \frac{e^{-f}}{(4\pi\tau)^{\frac{n+1}{2}}}$,
and note $\beta \mathbb{S}^n = N$.
\\

\begin{note}
	Here, we note $w$ such that $w^2 = e^{-f}$, it satisfies $\int_M w^2dv = (4\pi\tau)^{\frac{n+1}{2}}$.
\end{note} 

Using the following general formula for the entropy of a warped product :
\begin{align*}
\mathcal{W}^M(f,dr^2+h^2g_N,\tau) =& \int_0^{+\infty}h^n\int_N\left[\tau\left(4(\partial_r w)^2 +\frac{4|\nabla^Nw|^2+(\R^N-n(n-1)(h')^2-2n h h'')w^2}{h^2}\right)\right. \\
&\left.- w^2 \log(w^2)-(n+1) w^2\right](4\pi\tau)^{-\frac{n+1}{2}}dv dr,
\end{align*} 
we get :
\begin{align*}
\mathcal{W}^M(f,dr^2+h^2g_N,\tau) =& \int_0^{1}\left(\frac{r}{\beta}\right)^n\int_N\left[\tau\left(4(\partial_r w)^2 +\frac{4|\nabla^Nw|^2}{(\frac{r}{\beta})^2}\right)\right. \\
&\left.- w^2 \log(w^2)-(n+1) w^2\right](4\pi\tau)^{-\frac{n+1}{2}}dv dr\\
&+\int_1^{b}h^n\int_N\left[\tau\left(4(\partial_r w)^2 +\frac{4|\nabla^Nw|^2+(n(n-1)(\beta^{-2}-(h')^2)-2n h h'')w^2}{h^2}\right)\right. \\
&\left.- w^2 \log(w^2)-(n+1) w^2\right](4\pi\tau)^{-\frac{n+1}{2}}dv dr\\
&+\int_{h(b)}^{+\infty}r^n\int_N\left[\tau\left(4(\partial_r w)^2 +\frac{4|\nabla^Nw|^2+n(n-1)(\beta^{-2}-1)w^2}{r^2}\right)\right. \\
&\left.- w^2 \log(w^2)-(n+1) w^2\right](4\pi\tau)^{-\frac{n+1}{2}}dv dr,\\
&=(1)+(2)+(3)
\end{align*} 
where we have explicited the last term to show that it is corresponding to what we would get with $C(N)$.

 In the second term, for A big enough, the term involving $h''$ is like $A^{-1}$ times the term in $h'$. We will assume that A is big enough to forget about $h''$ (we only have strict inequalities in the end).
\\

 Let us consider $v$ a function on $C(N)$ such that :
$$\int_{C(N)}v^2(4\pi\tau)^{-\frac{n+1}{2}} = 1.$$
\\

 We are going to define $w=\phi(v)$ a function on $M$ such that :
$$\int_M w^2(4\pi\tau)^{-\frac{n+1}{2}} = 1.$$
\indent Where $\phi$ is a natural one to one correspondance. We will then prove that the entropy of $w$ on $M$ is smaller than the entropy of $v$ on $C(n)$.

 Define :
$$w(r,.) = \sqrt{h'(r)}v(h(r),.),$$
we have,
$$\int_M w^2(4\pi\tau)^{-\frac{n+1}{2}} = 1.$$

 And by the change of variable $\rho = h(r)$, we get :
\begin{align*}
\mathcal{W}^M&(f,dr^2+h^2g_N,\tau) = \int_0^{\frac{1}{\beta}}(\rho)^n\int_N\left[\tau\left(4(\partial_\rho v)^2 +\frac{4|\nabla^Nv|^2}{\rho^2}\right)\right. \\
&\left.- v^2 \left(\log(v^2)+n\log\beta\right)-(n+1) v^2\right](4\pi\tau)^{-\frac{n+1}{2}}dv d\rho\\
&+\int_{\frac{1}{\beta}}^{h(b)}\rho^n\int_N\left[\tau\left(4(\partial_\rho v)^2 +\frac{4|\nabla^Nv|^2+n(n-1)\left[\beta^{-2}-\left((h' o h^{-1}(\rho)\left(1+\mathcal{O}(\frac{1}{A})\right)v^2+\frac{1}{2}\log h'o h^{-1}\right)\right]}{\rho^2}\right)\right. \\
&\left.- v^2 \log(v^2)-(n+1) v^2\right](4\pi\tau)^{-\frac{n+1}{2}}dv d\rho\\
&+\int_{h(b)}^{+\infty}\rho^n\int_N\left[\tau\left(4(\partial_\rho v)^2 +\frac{4|\nabla^Nv|^2+n(n-1)(\beta^{-2}-1)v^2}{\rho^2}\right)\right. \\
&\left.- v^2 \log(v^2)-(n+1) v^2\right](4\pi\tau)^{-\frac{n+1}{2}}dv d\rho\\
&=(1)+(2)+(3).
\end{align*} 
\indent Now, what we would have got with the cone is :
\begin{align*}
\mathcal{W}^{C(N)}(v,dr^2+r^2g_N,\tau) =& \int_0^{\frac{1}{\beta}}(\rho)^n\int_N\left[\tau\left(4(\partial_\rho v)^2 +\frac{4|\nabla^Nv|^2+n(n-1)(\beta^{-2}-1)v^2}{\rho^2}\right)\right. \\
&\left.- v^2 \log(v^2)-(n+1) v^2\right](4\pi\tau)^{-\frac{n+1}{2}}dv d\rho\\
&+\int_{\frac{1}{\beta}}^{h(b)}\rho^n\int_N\left[\tau\left(4(\partial_\rho v)^2 +\frac{4|\nabla^Nv|^2+n(n-1)(\beta^{-2}-1)v^2}{\rho^2}\right)\right. \\
&\left.- v^2 \log(v^2)-(n+1) v^2\right](4\pi\tau)^{-\frac{n+1}{2}}dv d\rho\\
&+\int_{h(b)}^{+\infty}\rho^n\int_N\left[\tau\left(4(\partial_\rho v)^2 +\frac{4|\nabla^Nv|^2+n(n-1)(\beta^{-2}-1)v^2}{\rho^2}\right)\right. \\
&\left.- v^2 \log(v^2)-(n+1) v^2\right](4\pi\tau)^{-\frac{n+1}{2}}dv d\rho\\
&=(1')+(2')+(3').
\end{align*} 
\indent If we assume $\beta>1$, by taking the difference of these two expressions, we get :
$$\mathcal{W}^{M}(w,dr^2 + r^2g_N,\tau)-\mathcal{W}^{C(\beta\mathbb{S}^n)}(v,dr^2 + h^2g_N,\tau) >-(n\log \beta+\mathcal{O}(A)).$$
\indent In the case of a $\beta<1$, it is more convenient to compare to the Euclidean space. The computation is exacly the same, and we get :
$$\mathcal{W}^{M}(w,dr^2 + r^2g_N,\tau)-\mathcal{W}^{\mathbb{R}^n}(v,dr^2 + h^2g_N,\tau) >(n\log \beta+\mathcal{O}(A)).$$
\begin{rem}
It is crucial that the right term doesn't depend on $\tau$.
\end{rem}
\indent In both cases, if $\beta$ is close enought to $1$ and $A$ large enough in our parametrization, we have $\nu^M>\eta_n$
\end{proof}

\section{Renormalizations of the Ricci flow}
In this section we first introduce a few renormalization of the Ricci flow and discuss their different properties.

 We will then use them to, on the one hand smooth out some cones by manifolds with higher entropy and on the other hand to define a flow on the link only that increases the entropy of a cone.

 Let us define a few renormalizations of the Ricci flow. Some of them are motivated by the constance of a quantity, and others by the fact that some of perelman's quantities increase, at $\tau$ fixed.

 All of the flows introduced will act both on $g$ and a potential $f$. 

 We will consider renormalization of the following form :
$$
\left\{
\begin{array}{ll}
\partial_t g &= -2\left( \Ric+Hess f-\frac{\alpha}{n}g \right)\\
\partial_t f &= -\Delta f-\R+\alpha
\end{array}
\right.
$$
\indent We will be choosing some quite natural values for $\alpha$ (a lot more values of $\alpha$ give interesting properties to the flow)
\subsection{Evolution of some geometric quantities along the flows}
Let us compute the evolution of Perelman's quantities and other geometric quantities along such a flow.
\paragraph{\textbf{The scalar curvature}}
We have the following expression :
\begin{align}
\partial_t \R &= \Delta \R + 2\left<\Ric,\Ric-\frac{\alpha}{n}g\right> + \mathcal{L}_{\nabla f}\R\nonumber\\
&=\Delta \R + 2|\Ric|^2-2\frac{\alpha}{n}\R + \mathcal{L}_{\nabla f}\R \nonumber\\
&\geqslant \Delta \R + 2\frac{\R-\alpha}{n}\R + \mathcal{L}_{\nabla f}\R. \label{varR}
\end{align}
\paragraph{\textbf{The volume}}
We have the following variation of the volume :
\begin{align}
\partial_t\left(vol(N)\right) &= \int_N\left(\alpha-\R_{av}\right)dv \nonumber\\
&= \left(\alpha-\R_{av}\right)vol(M). \label{varvol}
\end{align}
Where $\R_{av}$ is the average value of the scalar curvature.
\paragraph{\textbf{Perelman's $\mathcal{F}$-functional}}
The derivative for $\mathcal{F}(f)$ is :
\begin{align}
\partial_t\left(\mathcal{F}(f(t),g(t))\right) &= \int_N\left<\Ric+Hess f-\frac{\alpha}{n}g,\Ric+Hess f\right>dv \nonumber\\
&=\int_N\left|\Ric+Hess f\right|^2dv-\frac{\alpha}{n}\mathcal{F}(f). \label{varF}
\end{align}
\paragraph{\textbf{Perelman's $\mathcal{W}$-functional at $\tau$ fixed}}
\begin{align}
\partial_t\left(\mathcal{W}(f(t),g(t),\tau)\right) &= \int_N\left<\Ric+Hess f-\frac{\alpha}{n}g,\Ric+Hess f-\frac{1}{2\tau}\right>dv\nonumber \\
&=\int_N\left|\Ric+Hess f\right|^2dv-\left(\frac{\alpha}{n}+\frac{1}{2\tau}\right)\mathcal{F}(f)+\frac{\alpha}{2\tau}. \label{varW}
\end{align}
\paragraph{\textbf{Inequalities among some natural quantities}}
:
\\
We have the following inequalities :
\begin{align*}
\R_{min}\leqslant \lambda^N\leqslant \R_{av}\leqslant \R_{max}\leqslant \frac{n}{2T_N},
\end{align*}
where we have noted $f_\mu$ a minimizer of $\mathcal{W}$.
\\
And note that we have equality if the manifold is a positively curved Einstein manifold.
\subsection{Volume preserving - $\alpha = \R_{av}$}
Here we recall the standard renormalization introduced by Hamilton.
\\
\indent It is a flow that preserves the volume of the whole manifold. It is useful to keep track of the limiting object thanks to its volume and some symmetries for example.
\\
\indent Along this flow,
$$
\left\{
\begin{array}{lllll}
\partial_t \R_{av} &\geqslant 0,\\
\partial_t vol &= 0,\\
\partial_t \R_{min}&\geqslant 2\frac{\R_{min}}{n}\left(\R_{min}-\R_{av}\right),\\
\partial_t \lambda &?\\
\partial_t \mu &?\\
\partial_t T_N&\geqslant 0.
\end{array}
\right.
$$
\subsection{Shrinking time preserving - $\alpha = \frac{n}{2T_N}$}
\begin{defn}
\item We will first consider Ricci flows that goes extinct in finite time getting close to a sphere to reduce our problem to the case of  cones over spheres that we have already taken care of. For that, we will define $T_N$ (where we have abusively noted $N = (N,g)$) such that :
\begin{align}
\left(\frac{1}{2(n-1)(T_N-t)}g_t\right) \xrightarrow[t\to T_N]{} g^{\mathbb{S}^n},\label{cvtoS}
\end{align}
where $t\mapsto g_t$ evolves according to the Ricci flow equation and starts at $g_0 = g$.
\end{defn}
This leads to considering the flow :
$$
\left\{
\begin{array}{ll}
\partial_t \hat{g} &= -2\left( \Ric+Hess \hat{f}-\frac{1}{2T_N}\hat{g} \right),\\
\partial_t \hat{f} &= -\Delta \hat{f}-\R+\frac{n}{2T_N}.
\end{array}
\right.
$$
The reason is the following :
\\
\indent Let us note $\hat{N}(t) = (N,\hat{g}(t))$ as an abusive notation, along the renormalized flow just introduced, the shrinking time is preserved :
\begin{lem}
Along the renormalized flow, for all $t$ :
$$T_{\hat{N}(t)} = T_{\hat{N}(0)}.$$
\begin{proof}
Let us consider $(N,g)$ such that \eqref{cvtoS} is satisfied.
\\
Let us consider two families of metrics starting at $g_0$ : 
\begin{itemize}
\item $(g_t)_t$ starting at $g_0$ evolving according to the Ricci flow equation.
\item $(\hat{g}_t)_t$ starting at $g_0$ evolving according to the renormalized Ricci flow equation.
\end{itemize}
By integrating the equation satisfied by $\hat{g}_t$ we have that :
\\
With $$\psi(t) = -T_N\log\left(1-\frac{t}{T_N} \right)\in[0,+\infty],$$
\begin{align}
\hat{g}_{\psi(t)} = \frac{1}{\left(1-\frac{t}{T_N}\right)}g_{t}.
\label{exphatt}
\end{align}
Choosing a $t_0$ and starting a Ricci flow $t\mapsto \tilde{g}^{t_0}_t$ at $\hat{g}_{t_0}$, 
By translation by $t_0$ and parabolic scaling by $\left(1-\frac{t}{T_N}\right)$.
\begin{align}
\left(\frac{1}{2(n-1)(T_N-t)}\tilde{g}^{t_0}_t\right) &= \frac{1}{2(n-1)(T_N-t)\left(1-\frac{t_0}{T_N}\right)}g_{t_0+\left(1-\frac{t_0}{T_N}\right)t} \nonumber \\
&=\left(\frac{1}{2(n-1)(T_N-\tilde{t})}g_{\tilde{t}}\right),\label{**}
\end{align}
with $\tilde{t} = t_0+\left(1-\frac{t_0}{T_N}\right)t$.
\\

Now, by hypothesis :
\begin{align*}
\left(\frac{1}{2(n-1)(T_N-\tilde{t})}g_{\tilde{t}}\right) \xrightarrow[\tilde{t}\to T_N]{} g^{\mathbb{S}^n},
\end{align*}
and by \eqref{**},
$T_{\hat{N}(t)} = T_N$.
\end{proof}
\end{lem}
\indent Along this flow,
$$
\left\{
\begin{array}{lllll}
\partial_t \R_{av} & ?\\
\partial_t vol &\geqslant 0\\
\partial_t \R_{min}&\geqslant 2\frac{\R_{min}}{n}\left(\R_{min}-\left(\frac{n}{2T_N}\right)\right)\\
\partial_t \lambda &?\\
\partial_t \mu(f_t,g_t,\tau) &\geqslant 0 \text{, if $\tau\leqslant\frac{n}{2\mathcal{F}(f_\mathcal{W})}\leqslant T_N$}\\
\partial_t T_N& = 0.
\end{array}
\right.
$$

\section{Blow down of type III immortal solutions with high $\nu$-functional}
We have been able to construct some nonsingular Ricci flows, that have a global curvature decay in $\frac{C}{t}$. We call them \textit{type III} solutions of the Ricci flow. 

 The goal would be to construct Ricci flows smoothing out cones thanks to them.

 For this purpose, we want to look at what is happening at large times, this is done by parabolically scaling down our Ricci flow and taking a limit.
\\
\begin{def}
\textit{Scaling down} a Ricci flow $(M,g(t))_t$ corresponds intuitively to sending every positive time to $+\infty$ while keeping the metric from being infinite. 

 Formally it is looking at the following sequence of Ricci flows (and its limit when $s\to \infty$ if it exists.)
\begin{align*}
g_s(t) = \frac{1}{s}g(st).
\end{align*}
\end{def}
\subsection{Hamilton's compactness theorem for Ricci flows}
Let us start by presenting Hamilton's compactness theorem for Ricci flows from \cite{ham} that will let us take (sub)limits of Ricci flows, and in particular of sequence of blowdowns under some assumptions.
\begin{thm}
Given $r_0 \in (0, \infty]$, let ${g_i(t)}_i$ be a sequence of Ricci flow solutions on connected \textit{pointed} manifolds $(M_i,g_i(t), m_i)_t$, defined for $t \in (A, B)$ with $-\infty \leqslant A < 0 < B \leqslant \infty$ (note that it is also possible to take a limit of intervals).
\\
\\
\indent We assume that for all i, $M_i$ equals the time-zero ball $B_0(m_i, r_0)$ and for all $r \in (0, r_0)$,
$\overline{B_0(m_i, r)}$ is compact. Suppose that the following two conditions are satisfied :
\begin{itemize}
\item For each $r \in (0, r_0)$ and each compact interval $I \subset (A, B)$, there is an $N_{r,I} < \infty$ so that
for all $t \in I$ and all $i$, 
$$\sup_{B_0(m_i ,r)\times I}| \Rm(g_i)| \leqslant N_{r,I}.$$
\item The time-$0$ injectivity radii at $m_i$ are bounded from below :
\\
There exists $\rho>0$ such that :
$$inj_{g_i(0)}(m_i)>\rho>0.$$

\end{itemize}
Then after passing to a subsequence, the solutions converge smoothly to a Ricci flow solution $g_\infty(t)$ on a connected pointed manifold $(M_\infty, m_\infty)$, defined for $t \in (A, B)$, for which :
\\
$M_\infty = B_0(m_\infty, r_0)$ and $\overline{B_0(m_\infty, r)}$ is compact for all $r \in (0, r_0)$. 

 That is, for any compact interval $I \subset (A, B)$ and any $r < r_0$, there are pointed time-independent diffeomorphisms
$\phi_{r,i} : B_0(m_\infty, r) \to B_0(m_i, r)$ so that ${(\phi_{r,i} \times Id)g_i}$ converges smoothly to $g_\infty$ on $B_0(m_\infty, r) \times I$.
\end{thm}
\subsection{Lower bound on the injectivity radius}
To look at what's happening at large times, we want to define a sublimit for 
\begin{align*}
(M,g_s(t),p)
\end{align*}
for some point $p$ in the manifold.
\\

 Thanks to the bound :
$$|\Rm_g|(.,t)\leqslant \frac{C}{t},$$
we have 
$$|\Rm_{g_s}|(.,t) = s|\Rm_g|(.,st)\leqslant s\frac{C}{st} = \frac{C}{t}.$$
\indent In particular, for positive times, there is no problem to get a uniform bound on $|\Rm_{g_s}|$.
\\

 Remains to get a positive lower bound on the injectivity radius. The injectivity radius is not a very convenient quantity to keep track of along a flow, so we are first going to recall a theorem by Cheeger, Gromov and Taylor in \cite{cgt} :
\begin{thm}
Let $(M, g)$ be a complete Riemannian manifold such that :
\begin{itemize}
\item There exists a constant $K\geqslant 0$ with 
$$|\Rm| \leqslant K,$$
\item There exists a point $p \in M$ and a constant $v_0 > 0$ such that
$$Vol_g(B_g(p, r)) \geqslant v_0r^n.$$
\end{itemize}
Then there exists a positive constant $i_0 = i_0(n,Kr^2,v_0)$ such that
$$inj_g(p)>i_0 r>0.$$
\end{thm}
Now, in our case, to apply Hamilton's compactness theorem. we would like to have 
$$\liminf_{t\to\infty} \left(\frac{inj_{g(t)}(p)}{\sqrt{t}}\right)  > 0.$$
\indent The particularity of our manifolds is that they have a quite large $\mu$-functional \textit{for all $\tau$}.
\\
\\
\indent So, adapting the proof of the first no local collapsing theorem of Perelman (thanks to the $\mu$-functional), we get :
\begin{lem}
Suppose that :
\begin{itemize}
\item There exists $A>0$ such that :
$$\nu^M(g(0)) > -A.$$
\item There exists $C>0$, such that, for all $t>0$,
$$|\Rm|(.,t)\leqslant \frac{C}{t}.$$
\end{itemize}
Then, there exists $v_0(A,C,n)>0$ such that :
\\
For all $t\leqslant 0$ :
$$Vol_{g(t)}(B_{g(t)}(p, \sqrt{t}))\geqslant v_0 t^{\frac{n}{2}}.$$
\begin{proof}
From the proof of Theorem 13.3 in the notes of Kleiner and Lott,
\\
\indent If $Vol_{g(t)}(B_{g(t)}(p, \sqrt{t}))\to 0$ when $t\to \infty$,
\\
\indent Then for a well chosen sequence $f_k$, such that $\int_M\frac{e^{-f_k}}{(4\pi t_k)^{\frac{n}{2}}} = 1$,
$$\mathcal{W}(f_k,g(t_k),t_k)\to -\infty.$$
\\
In particular, there would be $k_0$ such that :
\begin{align}
\mu(g(t_{k_0}),t_{k_0})\leqslant\mathcal{W}(f_{k_0},g(t_{k_0}),t_{k_0})\leqslant -A, \label{izi}
\end{align}
but, since $t\mapsto \mu\left(g(t),(2t_{k_0})-t\right)$ is nondecreasing, we get from \eqref{izi} that :
\begin{align*}
\mu(g(0),2t_{k_0})&\leqslant\mu(g(t_{k_0}),2t_{k_0}-t_{k_0})\\
&=\mu(g(t_{k_0}),t_{k_0})\\
&\leqslant -A,
\end{align*}
which contradicts the hypothesis on $\mu(\tau)$.
\end{proof}
\end{lem}
As a direct application of Cheeger-Gromov-Taylor theorem, we get :

\begin{cor}\label{curv et inj}
For any Riemannian manifold $(M,g(0))$ such that : 
\begin{itemize}
\item The Ricci flow starting at $(M,g(0))$ is a type III solution : 
\\
$\exists C>0$ such that : $$|\Rm(.,t)|\leqslant \frac{C}{t}.$$
\item There exists $A>0$, such that : 
$$\nu(g(0))\geqslant -A.$$
\end{itemize}
Then, defining : $g_s(t) = \frac{1}{s}g(st)$,
\\
we have for any $t_0>0$ : 
\begin{itemize}
\item $|\Rm_{g_s(t_0)}| \leqslant \frac{C}{t_0}$ (uniform in $s$),
\item $inj_{g_s(t_0)}\geqslant v_0 t_0^{\frac{n}{2}}$ where $v_0 = v_0(n,C,A)$ is coming from the last lemma (it is again uniform in $s$).
\end{itemize}
\end{cor}
\begin{rem}
As a consequence, for every type III solution with a uniform lower bound on their $\mu$-functional, one can extract a scaling down sublimit by Hamilton's compactness theorem.

The corollary of the next part states this fact more precisely in our case.
\end{rem}
\subsection{Scaling down manifolds of large $\mu$-functional}
Let us state a corollary of the global pseudolocality coupled with Hamilton's compactness theorem :
\begin{prop}
For any Riemannian manifold $(M,g(0))$ such that 
$$\nu^M(g(0))> -\eta_n,$$

 there exists an immortal type III solution of the Ricci flow starting at $(M,g(0))$ : $t \mapsto g(t)$.
\\

And the scaling down sequence : 
\begin{align*}
(M,g_k(t),p)_k = \left(M,\frac{1}{k}g_k(kt),p\right)_k.
\end{align*}
\indent Has a sublimit which is an immortal type III solution of the Ricci flow.
\begin{rem}
Moreover, we have :
\begin{enumerate}
\item If $(M,g_0)$ is asymptotic to a cone $C(N)$, then the sublimit of blow downs is a Ricci flow coming out of the cone $C(N)$ and asymptotic to it at all times, see section 5.1 and 5.2 in \cite{lz} (where it is possible to use the global pseudolocality rather than the usual pseudolocality result).
\item If $\Ric\geqslant 0$ along the flow, then the blowdown sublimits are gradient expanding solitons by \cite{Ma}.
\end{enumerate}
\end{rem}
\begin{proof}
Thanks to the global pseudolocality pseudolocality, we know that a Ricci flow exists for all positive times, and is a type III solution of the Ricci flow.
Let us first take a sublimit by Hamilton's compactness theorem :
\\
\indent The two hypothesis to check to take a sublimit in the sense of Hamilton are satisfied thanks to the corollary \ref{curv et inj}.
\\
\indent So we get a sublimit $(M_\infty,g_\infty)$. If the Ricci flow was asymptotic to a cone, thanks to the section 5 of \cite {lz}, the sublimit flow comes out of the cone.
\begin{rem}
Note that if we had an actual limit, we could argue that we have an expanding soliton in the limit.
\end{rem}
\end{proof}
\end{prop}

\end{document}